\documentclass[11pt,a4paper]{amsart}
\usepackage{a4wide}
\usepackage{amsfonts,amsthm,amsmath,amssymb}
\usepackage{amscd,color}
\usepackage{verbatim}
\usepackage{graphicx,subfigure}
\usepackage{pgfplots} %%per a tiksfigure

\newtheorem*{thmA}{Theorem A}

\newtheorem*{Conj}{Conjecture}

\theoremstyle{plain}

\newtheorem{theorem}{Theorem}[section]
\newtheorem{lemma}[theorem]{Lemma}

\newtheorem{prop}[theorem]{Proposition}

\newtheorem{remark}{Remark}

\newsavebox{\savepar}

\graphicspath{{./figures/}}

\begin{document}

\title[Traub's method as a dynamical system]{On the basins of attraction of
a one-dimensional family of root finding algorithms: From Newton to Traub.}

\date{}

\author{Jordi Canela}
\address{Institut Universitari de Matem\`atiques i Aplicacions de Castell\'o and Departament de Matem\`atiques, Universitat Jaume I, 12071 Castell\'o de la Plana, Spain. ORCID: https://orcid.org/0000-0001-7879-5438}
\email{canela@uji.es}
\author{Vasiliki Evdoridou}
\address{School of Mathematics and Statistics, The Open University,	Walton Hall, Milton Keynes MK76AA, UK. ORCID: https://orcid.org/0000-0002-5409-2663}
\email{vasiliki.evdoridou@open.ac.uk }
\author{Antonio Garijo}
\address{Departament d'Enginyeria Inform\`atica i Matem\`atiques,
Universitat Rovira i Virgili, 43007 Tarragona, Catalonia, Spain. ORCID:  https://orcid.org/0000-0002-1503-7514}
\email{antonio.garijo@urv.cat}
\author{Xavier Jarque}
\address{Departament de Matemàtiques i Informàtica, Universitat de Barcelona, Gran Via, 585, 08007 Barcelona, Catalonia and  Centre de Recerca Matemàtica, Edifici C, Campus Bellaterra, 08193 Bellaterra, Catalonia.
ORCID:  https://orcid.org/0000-0002-6576-9780}
\email{xavier.jarque@ub.edu}

\thanks{\textbf{Statements and declarations:} The first author was supported by the Spanish Ministry of Economy and Competitiveness through the Mar\'ia de Maeztu Programme for Units of Excellence in	R\&D (MDM-2014-0445), by BGSMath Banco de Santander Postdoctoral 2017, and by the project UJI-B2019-18 from Universitat Jaume I. The second author was supported by the London Mathematical Society, the IMUB and the EPSRC grant EP/R010560/1. The third author was supported by PID2020-118281GB-C33. The first and fourth authors were supported by  PID2020-118281GB-C32.}\

\begin{abstract}
In this paper we study the dynamics of damped Traub's methods $T_\delta$ when applied to polynomials. The family of 
damped Traub's methods consists of root finding algorithms which contain both Newton's ($\delta=0$) and Traub's method ($\delta=1$). 
Our goal is to obtain several topological properties of the basins of attraction of the roots of a polynomial $p$ under $T_1$, which are used to determine
a (universal) set of initial conditions for which convergence to all roots of $p$ can be guaranteed. We also numerically explore the global properties of the dynamical plane for $T_\delta$ to better understand the connection between Newton's method and Traub's method.   \newline


\end{abstract}

\maketitle

%
%
%
%------------------------INTRODUCTION ----------------------------
%
%
%
\section{Introduction}

Dynamical Systems is a powerful tool which allows us to obtain a deep understanding of the global behaviour of  {\it root-finding} algorithms, that is, iterative methods capable to  determine numerically the solutions of the (non-linear) equation $f(x)=0$. In most cases, the order of  convergence of those methods near the zeros of $f$ is well known, but the behaviour and effectiveness when initial conditions are chosen on the whole space is in general unclear. This is, in fact, the main problem to tackle when studying a root-finding algorithm from the dynamical systems point of view, particularly, when we do not know a priori {\it where the roots are}, when there are many roots, or when we do not know how they are distributed. 

The numerical exploration of the solutions of the equation $f(x)=0$ has always been   a central problem in many areas of applied mathematics, from biology to engineering, since most mathematical models require a thorough knowledge of the solutions of certain equations. Once we are certain that no algebraic manipulation of the equation will allow to explicitly find out the  solutions, one can try to build numerical methods which will approximate the solutions with arbitrary precision. Perhaps the most well-known and universal method is {\it  Newton's method} (see for instance \cite{Newton-Blanchard}) inspired by the linearisation of the equation $f(x)=0$ but also other methods like the one under consideration here, {\it Traub's method} (see \cite{Traub_Book}), have been shown to be efficient  when converging.

Roughly speaking, all these iterative methods give efficient ways to find the solutions of $f(x)=0$, at least once you have a good approximation of them. However, there is a significant amount  of uncertainty  when the initial conditions are freely chosen, when, as we mentioned above, there is no {\it natural} candidate for the solution, or the number of solutions of $f(x)=0$ is large. It is  in this context where dynamical systems can play a central role since a precise description of the dynamical plane could be a cornerstone input. More precisely, the topological properties of the (immediate) attracting basins of the points that correspond to solutions of $f(x)=0$ are the key tool to elaborate 
algorithms which will calculate all solutions {\it at once} in an efficient way.

A paradigmatic example of this method  can be found in the seminal paper  by Hubbard, Schleicher and Sutherland, \cite{HowToNewton}, where the authors first prove theoretical results on the topology of the mentioned invariant sets for Newton's method and then they use this information to create efficient algorithms to find {\it all} solutions, even in the case that the degree of $p$ is extremely high. Let us summarize the  main result in \cite{HowToNewton}. First, we introduce the basic notation from holomorphic dynamics.

Let $R:\hat{\mathbb C} \to \hat{\mathbb C}$ be a rational map. A point $z=\xi$ is fixed if $R(\xi)=\xi$ (resp.\ periodic of period $p$ if $R^p(\xi)=\xi$  for some $p\geq 1$). The multiplier of $\xi$ is $\lambda=R'(\xi)$ (resp.\ $\lambda=(R^p)'(\xi)$). The fixed or periodic point $\xi$ is \textit{attracting} if $|\lambda|<1$ (\textit{superattracting} if $\lambda=0$),  \textit{repelling} if $|\lambda|>1$, and \textit{indifferent} if $|\lambda|=1$. If $\xi$ is attracting, we define the {\it basin of attraction of} $\xi$ as
$$
\mathcal A_R(\xi):=\mathcal A(\xi)=\{z\in \hat{\mathbb C} \ | \ R^n(z)\to \xi,\ n\to \infty  \}.
$$
In what follows we  omit  the dependence with respect to the rational map under consideration, unless it is mandatory. It is easy to see that $\mathcal A(\xi)$ is an open set containing $\xi$. We denote by $\mathcal A^{\star}(\xi)$ the connected component of $\mathcal A(\xi)$ containing $\xi$.

Recall that if  $p$ is a polynomial of degree $d\geq 2$, the rational map defined as
\begin{equation}\label{eq:newton}
N_p(z):= z- \frac{p(z)}{p'(z)}
\end{equation}
is known as Newton's map (or method) applied to $p$. The map $N_p$ is the {\it universal} root-finding algorithm and it satisfies a key global dynamical property: the point $z=\alpha$ is a root of $p$ if and only if it is an attracting fixed point of $N_p$. In fact, if $z=\alpha$ is a simple root of $p$, then $N_p^{\prime}(\alpha)=0$ ($z=\alpha$ is superattracting) and  the local order of convergence of $N_p$ near $z=\alpha$ is quadratic.  We remark that $z=\infty$ is always a repelling fixed point of $N_p$.

We now turn to the dynamics of the map $N_p$. It is clear that for all initial conditions $z_0\in \mathbb C$ such that
$$
z_0\in \bigcup_{j=1}^d \mathcal A\left(\alpha_j\right)
$$ 
the sequences $\{N_p^n(z_0)\}_{n\geq 0}$ will converge to one of the roots of $p$. Newton's map might have other {\it stable} periodic components not related to the roots of $p$ (see \cite[Proposition 4]{Smale} and \cite{FamiliesRational}). Those open domains on the dynamical plane determine a positive measure set of {\it bad} initial conditions that we want to avoid when finding all roots of $p$. But, of course, a priori there is no control on the topology and distribution of these domains in the plane and so it seems difficult to choose the initial conditions carefully enough so that they are not in the positive measure bad set (if any). The authors in \cite{HowToNewton} constructed a universal and explicit set of initial conditions $\mathcal S_d$ (only depending on the degree of the polynomial) such that for any given $z=\alpha_j,\ j=1,\ldots, d,$ at least one of the corresponding sequences $\{N_p^n(z_0),\ z_0 \in \mathcal S_d\}_{n\geq 0}$,  converges to $\alpha_j$. The existence of the set $\mathcal S_d$ is guaranteed by the following key properties of the immediate basins of attraction for the fixed points of $N_p$.

\begin{theorem}[]\label{theorem:HSS}
Let $p$ be a polynomial of degree $d\geq 2$. Assume that $p(\alpha)=0$ and let $N_p$ be the corresponding Newton's map. Then $\mathcal A^{\star}(\alpha)$ is a simply connected unbounded set.
\end{theorem}

The above  result was proven by Przytycki  \cite{Prz}. Later on, Shishikura \cite{ConnectivityJulia}, generalized the simple connectivity of $\mathcal A^{\star}(\alpha)$ (and, in fact, of any Fatou component)  by proving that the Julia set of any rational map having one, and only one, weakly repelling fixed point is connected. For a general overview of this topic there are many excellent references, see for instance \cite{MilnorBook, BeardonBook,CarlesonGamelinBook}. For  concrete results on Newton's method as a dynamical system see \cite{HowToNewton, ConnectivityMero, CombinatoricsNewton, TanLei2, LocalConnectivityNewton}.

This paper is a step forward in order to extend the above theorem to a  class of root-finding algorithms which includes Newton's method as well as Traub's method. But the underlying motivation is to be able to construct  an $\mathcal S_d$ like-set for these root--finding algorithms which we can use to find all  the roots of $p$ at once. More precisely, we consider the \textit{damped Traub's family} of  root-finding algorithms
\begin{equation}\label{eq:traub_delta}
T_{p,\delta}(z)=N_p(z)-\delta\ \frac{p(N_p(z))}{p'(z)},
\end{equation}
which depends on a complex parameter $\delta$. To our knowledge this family, with $\delta\in \mathbb R$, was first considered in \cite{DampedTraubQuadratic, DampedTraubCubic}. In these papers the authors analysed the existence of attracting cycles other than the roots of the polynomial $p$, when $p$ is a quadratic or cubic polynomial.
Notice that $\delta=0$ corresponds to Newton's map. Traub (see \cite{Traub_Book}) proposed the root-finding algorithm, nowadays known as Traub's  method, which corresponds to $\delta=1$. For simple roots of $p$ the (local) order of convergence of Traub's method is cubic but it is worth to be noticed that each Traub's iteration is {\it one and a half} iterations for Newton's method.
  More precisely, two iterates of Newton's method require evaluating both $p$ and $p'$ at $z$ and $w:=N_p(z)$, while one iterate of Traub's method requires evaluating $p$ at $z$ and $w$ and $p'$ only at $z$.  
In other words, there are clear local advantages of Traub's method (cubic instead of quadratic) but it has also some drawbacks compared to Newton's method as it requires some extra evaluations and it might have attracting fixed points which do not correspond to any root of $p$. See Figure \ref{fig:One_Attractor}.

In any event, proving  an equivalent result to Theorem \ref{theorem:HSS} for $T_{p,1}$ will provide the tools for constructing the $\mathcal S_d$ like-set of {\it good} initial conditions with the advantage of the cubic, instead of quadratic, convergence of Traub's method. Nonetheless, according to some rigorous arguments plus the numerical experiments we have done we state the following conjecture.
 
\begin{Conj}
Let $p$ be a polynomial of degree $d\geq 2$. Assume that $p(\alpha)=0$ and let $T_{p,1}$ be the corresponding Traub's map. Then $\mathcal A^{\star}(\alpha)$ is a simply connected unbounded set.  
\end{Conj}

Neither Przytycki's  nor Shishikura's  proof of Theorem \ref{theorem:HSS} apply to the family $T_{p,\delta}$ except for $\delta=0$. Indeed, their proofs rely on the fact that for Newton's method there are no finite fixed points other than the roots.  In this paper we prove the conjecture in the case where some additional hypothesis holds.

\begin{figure}[hbt!]
	\centering
	\begin{tikzpicture}
		\begin{axis}[width=215pt, axis equal image, scale only axis,  enlargelimits=false, axis on top]
			\addplot graphics[xmin=-0.75,xmax=0.75,ymin=-0.75,ymax=0.75] {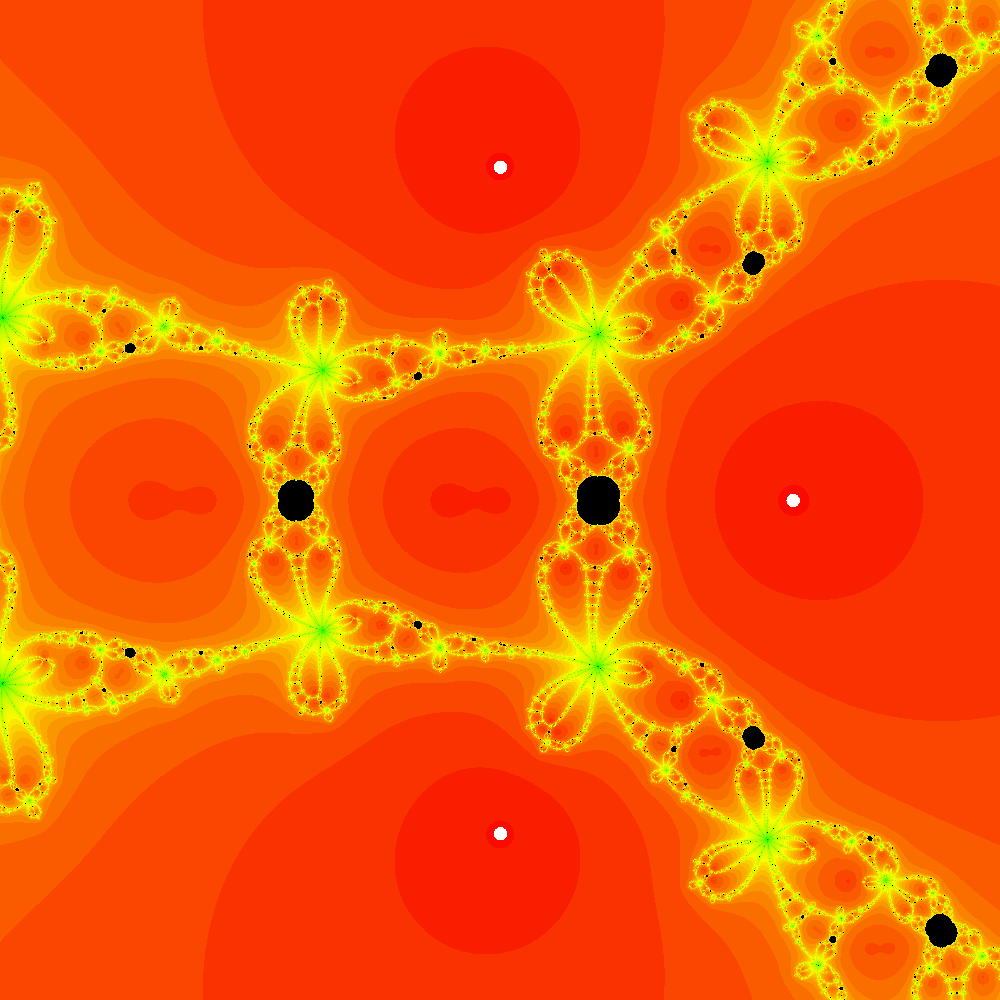};
		\end{axis}
	\end{tikzpicture}
	\put(-42,96){{ $0.439$}}
	\put(-108,174){{\small  $0.5i$}}
		\put(-115,34){{\small  $-0.5i$}}
	\put(-76.5,106){\tiny{\color{white} $\bullet$}}
%	\put(-90,106){ {\color{white} $\rightarrow$}}
		\put(-85,106){\scriptsize  {$\zeta$}}
		   \hglue 0.3truecm \includegraphics[width=215pt]{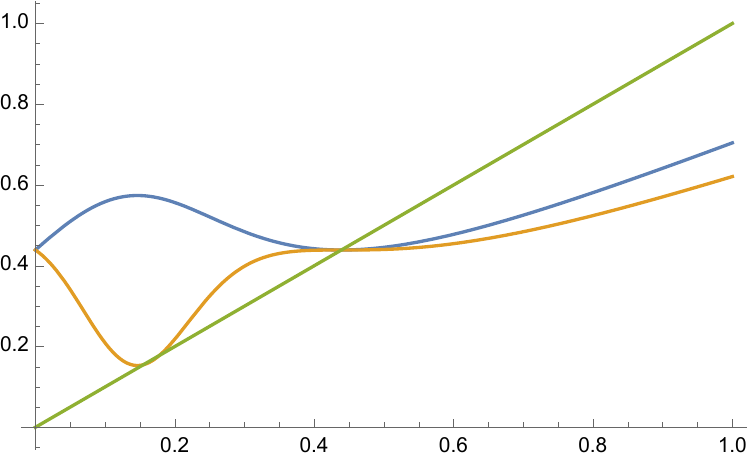}
	\put(-180,81){\small  $N_p(x)$} 
	\put(-40,64){\small  $T_p(x)$} 
	\put(-42,120){\small  $y=x$} 
	\caption{In the left picture we illustrate the dynamical plane of Traub's method applied to the cubic polynomial $p(z)=(z^2+0.25)(z-0.439)$. In this case, $T_p$ has an attracting fixed point located at $\zeta  \approx 0.155$.  The basins of attraction of the three fixed points associated with the zeros of $p$ are shown in red. The basin of attraction of $\zeta$ is shown in black. In the right picture we illustrate the Newton's and Traub's maps restricted to $\mathbb R$. We can observe that $N_p|_{\mathbb{R}}$ has one attracting fixed point while $T_p|_{\mathbb{R}}$ has two attracting fixed points. } 
	\label{fig:One_Attractor}
\end{figure}
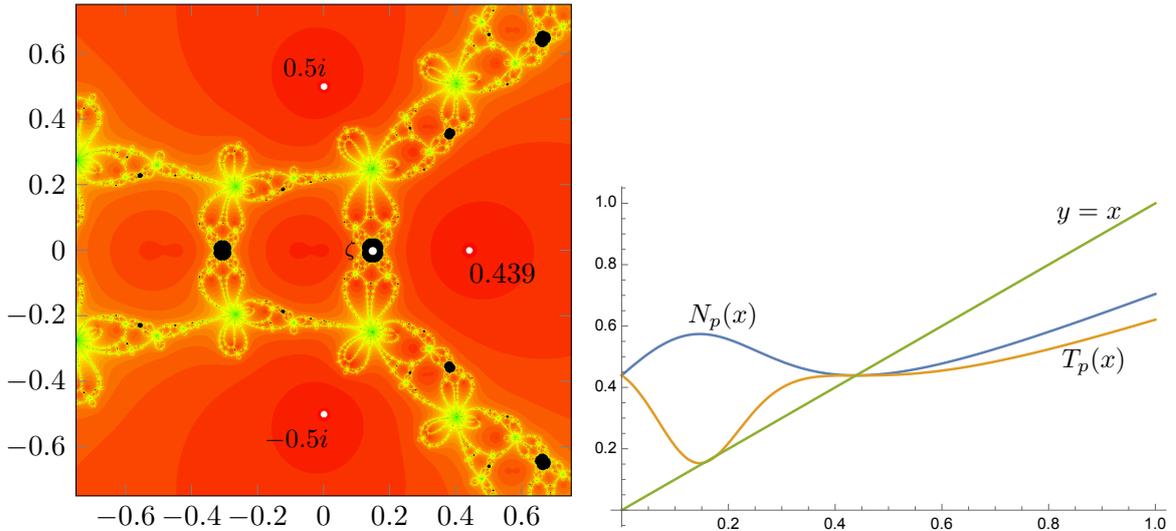

\begin{thmA}
Let $p$ be a polynomial of degree $d\geq 2$. Assume that $p$ satisfies one of the following conditions:
\begin{itemize}
\item[(a)] $d=2$, or
\item[(b)]  it can be written in the form $p_{n,\beta}(z):=z^n-\beta$ for some $n\geq 3$ and $\beta\in\mathbb C $. 
\end{itemize} 
Suppose that $p(\alpha)=0$ and consider $T_{p,\delta}$ with $\delta  \in [0,1]$. Then  $\mathcal A_{\delta}^{\star}(\alpha)$ is a simply connected, unbounded set.  
\end{thmA}

We strongly believe, and it is indicated by our numerical experiments, that working with the $\delta$-family instead of proving the conjecture for $T_{p,1}$ as an isolated map, might have important advantages. See Section \ref{sec:numconcl} for details.

In Section \ref{sec:properties} we show the main properties of the family of rational maps $T_{p,\delta}$. Section \ref{sec:quadratic}  is devoted to prove Theorem A  (a). Moreover, we also study in this section the $\delta-$parameter plane of $T_{p,\delta}$ and we prove that $\delta=0$ and $\delta=1$ belong to the same hyperbolic component of the parameter plane. In Section \ref{sec:pol} we prove Theorem A (b). Finally, in Section  \ref{sec:numconcl}  we discuss some numerical evidence supporting the conjecture as well as illustrating that some of the arguments used in the previous sections will not work in general; so other approach will be needed.

\section*{Acknowledgments}
We are grateful to the anonymous referee for all his/her suggestions which have clearly improved the previous version of this paper. We also thank the Institute of Mathematics at Universitat de Barcelona (IMUB) for the hospitality during the visit of the second author, when this project started.

%\ackknowlegment{ }

%
%
%
%------------------------LOCAL. PROPERTIES ----------------------------
%
%
%

\section{Local dynamics of the family $T_{p,\delta}$} \label{sec:properties}

In this section we prove fundamental dynamical properties for maps in $T_{p,\delta}$.  These properties will allow us to prove Theorem A. First, the local behaviour of $T_{p,\delta}$ near the roots of $p$ is described (Lemma \ref{multiple root}). Afterwards, we study the dynamical behaviour of the maps $T_{p,\delta}$ at $z=\infty$ by proving under which conditions 
infinity is an attracting fixed point (Lemma \ref{lem:infinity}). Finally, we study the critical points of $T_{p,\delta}$ (Lemma  \ref{lemma:critics}),  which play a key role in the local and global dynamics. 

 We can assume that  $p(z)$ is a monic polynomial since easily we have that  $N_{\lambda p } = N_{p}$ and $T_{\lambda p,\delta}=T_{p,\delta}$ for all $\lambda \in \mathbb C$.   Throughout this section we will use either of the following equivalent expressions for the polynomial $p$:
\begin{equation}\label{eq:p}
p(z)= z^d+ \dots +a_1z+a_0= \prod_{i=1} ^d (z- \alpha_i).
\end{equation}

\noindent Notice, in particular, that we are not assuming that the zeros of $p$ are simple. In Lemma \ref{multiple root} we describe the local behaviour of the map $T_{\lambda p,\delta}$ near the roots of $p$. This study depends on whether the root is simple or multiple. This local study was previously done in \cite[Theorem 1]{DampedTraubQuadratic} (for the case of simple roots) and \cite[Theorem 1]{DampedTraubCubic} (for the case of multiple roots). Even though Lemma~\ref{multiple root} follows directly from \cite[Theorem 1]{DampedTraubQuadratic} and \cite[Theorem 1]{DampedTraubCubic}, we add the proof for the sake of completeness. The main difference between these results is that \cite[Theorem 1]{DampedTraubQuadratic} and \cite[Theorem 1]{DampedTraubCubic} are obtained by studying the corresponding error terms of the numerical methods while in Lemma~\ref{multiple root}  we use the multiplier of the roots as fixed points of $T_{p,\delta}$.

\begin{lemma} \label{multiple root}
Let $p$ be a monic polynomial of degree $d \geq 2$ which has a root $\alpha$ of multiplicity $1\leq k  \leq d,$ i.e. $p(z)= (z- \alpha)^k q(z),$ where $q$ is a polynomial of degree $d-k$ with $q(\alpha) \neq 0$. Then, $\alpha$ is a fixed point of $T_{p,\delta}$. Moreover the following statements hold. 
\begin{itemize}
\item  If $k=1$ (simple roots) then $T_{p,\delta}$ has a superattracting fixed point at $\alpha$ for all $\delta \in \mathbb C$. 
\item  If $k \geq 2$ (multiple roots) then $T_{p,\delta}$ has an attracting fixed point at $\alpha$ if and only if  
\[
\delta\in \mathbb D_k:=\left\{    z \in  \mathbb C \, ; \, \left |z - \frac{ k^k }{(k-1)^{k-1}} \right | < \frac{k^{k+1}}{(k-1)^k } \right \}.
 \]
\end{itemize}

\end{lemma}
\begin{proof}
To simplify notation we  omit  the dependence on $p$; that is, $T_\delta=T_{p,\delta}$.  Let $\alpha$  be a root of $p$ and assume $p(z)= (z- \alpha)^k q(z),$ where $q$ is a polynomial of degree $d-k$ with $q(\alpha) \neq 0$. We have 
\begin{equation}\label{eq:new_tdelta}
T_{\delta}(z)= z-\frac{p(z)}{p^{\prime}(z)}-\delta\;\frac{\left(N(z)-\alpha\right)^kq\left(N(z)\right)}{p'(z)}.
\end{equation}
Some computations show that
\begin{equation} \label{eq:new_tdelta_2}
\begin{split}
&q\left(N(z)\right)=q\left(\frac{kzq(z)+(z-\alpha)(zq^{\prime}(z)-q(z))}{kq(z)+(z-\alpha)q^{\prime}(z)}\right):=q(A) \\
&\left(N(z)-\alpha\right)^k=(z-\alpha)^k\left(\frac{(k-1)q(z)+(z-\alpha)q^{\prime}(z)}{kq(z)+(z-\alpha)q^{\prime}(z)}\right)^{k}.
\end{split}
\end{equation}
From \eqref{eq:new_tdelta} and \eqref{eq:new_tdelta_2} and some further computations we get
\begin{equation} \label{eq:new_tdelta_3}
T_\delta(z)=z-B_1(z)(z-\alpha)-\delta B_2(z)(z-\alpha)q(A), \\
\end{equation}
where 
\begin{equation}\label{eq:b1b2} 
\begin{split}
&B_1(z)=\frac{q(z)}{kq(z)+(z-\alpha)q^{\prime}(z)} \\
&B_2(z)=\frac{\left((k-1)q(z)+(z-\alpha)q^{\prime}(z)\right)^k}{\left(kq(z)+(z-\alpha)q^{\prime}(z)\right)^{k+1}}.
\end{split}
\end{equation}
Now, trivially, $T_\delta(\alpha)=\alpha$. Taking the derivative
in \eqref{eq:new_tdelta_3} we have
$$
T^{\prime}_\delta(z)=1-B_1(z)-(z-\alpha)B_1^{\prime}(z)-\delta\left(B_2(z)-(z-\alpha)B_2^{\prime}(z)\right).
$$
Evaluating at $z=\alpha$ and using the expressions of $B_1(z)$ and $B_2(z)$ in \eqref{eq:b1b2}, we get
 $$
 T_{\delta}'(\alpha)= \frac{k-1}{k} - \delta \left(\frac{k-1}{k}\right)^k \frac{1}{k}.
 $$

Thus, if $k=1$ then $\alpha$ is a superattracting fixed point of $T_{\delta}$ since $T_{\delta}'(\alpha)=0$. If $k\geq 2$ then $\alpha$ is an attracting fixed point  of $T_{\delta}$ if and only if $ | T_{\delta}'(\alpha)|<1$ and the result follows since
\[
 | T_{\delta}'(\alpha)|<1 \iff \left |  \frac{k-1}{k} - \delta \left(\frac{k-1}{k}\right)^k \frac{1}{k}  \right | < 1 \iff   \left |\delta - \frac{ k^k }{(k-1)^{k-1}} \right | < \frac{k^{k+1}}{(k-1)^k }.
\]

 \end{proof}

From the lemma above we conclude that if $\alpha$ is a simple root of $p$, the local order of convergence of any map in $T_{p,\delta}$ is at least quadratic (a well known result for Newton's maps corresponding to $\delta=0$).  Later in this section we prove that for  $\delta=1$ (Traub's method) the local order of convergence is at least cubic (see  Lemma \ref{lemma:critics}~(a)); so near simple roots of $p$ Traub's method is more efficient than Newton's method. It is well known that for $\delta=0$, the point  $z=\infty$ is always a repelling fixed point with multiplier $N_p'(\infty)=d/(d-1)$. However the dynamical behaviour of $z=\infty$ for maps in $T_{p,\delta}$ depends on the particular choice of $\delta$. The next lemma shows that the point $z=\infty$ is fixed except for the degeneracy parameter $\delta=d^d/(d-1)^{d-1}$.  In particular,  in the case that  infinity is an attracting fixed point we obtain  open sets of initial conditions for which $T_{p,\delta}$ does not converge to the roots of $p$. Such parameters are avoided when looking for roots of a degree $d$ polynomial. 
%
% It also describes for which parameters infinity is an attracting fixed point and, hence, leads to open sets of initial conditions which do not converge to the roots of $p$ under iteration of $T_{p,\delta}$. Those are parameters are to be avoided when looking for roots of a degree $d$ polynomial. 
%

\begin{lemma} \label{lem:infinity}
Let $p$ be a monic polynomial of degree $d$, and assume $\delta\in \mathbb C$, $\delta\neq d^d/(d-1)^{d-1}$. Then, $z=\infty$ is a fixed point of  $T_{p,\delta}$ and satisfies:
\begin{enumerate}
	\item[(a)] it is repelling if 
	\[
	 0 < \left|\delta - \frac{ d^d }{(d-1)^{d-1}} \right | < \frac{d^{d+1}}{(d-1)^d },
	\]
		\item[(b)] it is attracting if 
	\[
	\left|\delta - \frac{ d^d }{(d-1)^{d-1}} \right | > \frac{d^{d+1}}{(d-1)^d },
	\]
		\item[(c)] and it is indifferent otherwise. 
\end{enumerate}

 \end{lemma}

\begin{proof}
As before we simplify notation by erasing the dependence on $p$.  Since $T_{\delta}$ is a rational map we rewrite the damped Traub's map as a quotient of two polynomials.  We claim that 
\[
T_{\delta}(z)= z - \frac{p(z)}{p'(z)} - \delta \frac{[p(z)]^2 r(z)}{[p'(z)]^{d+1}},
\]
\noindent where $r(z)$ is a polynomial of degree $d^2-2d$. To see the claim we notice that from \eqref{eq:p} we have
$$
p'(z)=\displaystyle \sum_{i=1}^d p_i(z), \quad {\rm where} \quad p_i(z)=\prod_{k=1\,,\,k\neq i}^d (z-\alpha_k).
$$ 
Also,
\begin{eqnarray*}
p(N_p(z)) &=&  \prod_{i=1}^d \left(z- \frac{p(z)}{p'(z)}- \alpha_i \right) =\frac{1}{[p'(z)]^d} \prod_{i=1}^d [(z-\alpha_i)p'(z)-p(z)] \nonumber \\
&=& \frac{1}{[p'(z)]^d} \prod_{i=1}^d [(z-\alpha_i) \sum_{k=1}^d p_k(z)-p(z)] =\frac{1}{[p'(z)]^d} \prod_{i=1}^d (z-\alpha_i) \sum_{k=1, k \neq i}^d p_k(z) , \nonumber  \\
&=& \frac{1}{[p'(z)]^d} \prod_{i=1}^d (z-\alpha_i)^2 \sum _{k=1\,,k \neq i}^dp_{k,i}(z) =  \frac{[p(z)]^2}{[p'(z)]^d} \prod_{i=1}^d \sum _{k=1\, , k \neq i}^dp_{k,i}(z), \nonumber \\
\end{eqnarray*}
\noindent where $p_{k,i}(z)=\prod_{j=1\,,\,j\neq i,k }^d (z-\alpha_j)$ and  $r(z)=\displaystyle \prod_{i=1}^d \sum _{k=1\, , k \neq i}^dp_{k,i}(z)$ is a polynomial of degree $d^2 -2d$. Therefore,  we obtain

\begin{equation}\label{eq:rat_T_lambda}
T_{\delta}(z)= \frac{z [p'(z)]^{d+1}-p(z)[p'(z)]^d-\delta [p(z)]^2r(z)}{[p'(z)]^{d+1}}.
\end{equation}

We now compute the leading coefficients of the numerator and the denominator in \eqref{eq:rat_T_lambda}. Since the polynomial $p$ has degree $d$ with leading coefficient equal  to 1, $p'$ has degree $d-1$ with leading coefficient equal to $d$ and the polynomial $r(z)$ has degree $d^2-2d$ with leading coefficient equal to $(d-1)^d$, some computations show that  
\begin{equation}\label{eq:rat_T_lambda_2}
T_{\delta}(z)= \frac{ [ d^{d+1} - d^d - \delta (d-1)^d ] z^{d^2} + \cdots }{ d^{d+1} z^{d^2-1} + \cdots }\ .
\end{equation}
Consequently, if 
$$
\delta \neq \frac{d^{d+1}-d^d}{(d-1)^d} = \frac{d^d}{(d-1)^{d-1}}
$$ 
the numerator of the rational map $T_{\delta}$ has degree $d^2$ while the denominator has degree $d^2-1$. We conclude that $T_{\delta}(\infty)=\infty$. On the other hand, if $\delta=d^d/(d-1)^{d-1}$, the degree of the denominator of $T_{\delta}$ is $d^2-1$ and the degree of the numerator is at most  $d^2-1$, so $T_{\delta}(\infty)\neq\infty$.

Assume $T_{\delta}(\infty)=\infty$. To compute the multiplier at $z=\infty$ we use the auxiliary map $G_{\delta}(z)= 1/T_{\delta}(1/z)$. Straightforward computations from \eqref{eq:rat_T_lambda_2} give 
\[
G_{\delta}(z) =\frac{1}{T_{\delta}(1/z)}= z \ \frac{d^{d+1} + \cdots}{[ d^{d+1} - d^d - \delta (d-1)^d ]+ \cdots },
\]
and so  
\[
T'_{\delta}(\infty)=G'_{\delta}(0)=\frac{d^{d+1}}{\left(d-1\right)\left[d^d-\delta\left(d-1\right)^{d-1}\right]}.
\]

We know that $z=\infty$ is a repelling fixed point of $T_{\delta}$ if and only if $|T'_{\delta}(\infty)|>1$. Hence,  statement (a) follows from 
\[
|T'_{\delta}(\infty)|>1 \iff \left |  \frac{d^{d+1}}{\left(d-1\right)\left[d^d-\delta\left(d-1\right)^{d-1}\right]} \right | > 1 \iff   \left |\delta - \frac{ d^d }{(d-1)^{d-1}} \right | < \frac{d^{d+1}}{(d-1)^d }.\]
Statements (b) and (c) follow similarly by imposing $|T'_{\delta}(\infty)|<1$ and $|T'_{\delta}(\infty)|=1$, respectively.
 \end{proof}

We finish this section studying the \textit{critical points} of the maps $T_{p,\delta}$. The importance of these points comes from the fact that they are related to the stable behaviour of the dynamics of any rational map. In order to explain this relation, let us first recall the basic concepts we need here about rational dynamics.
Given a rational map $R: \hat{\mathbb C} \to \hat{\mathbb C}$,  we consider the dynamical system given by the iterates of $R$. The Riemann sphere splits into two completely $R-$invariant subsets: the {\it Fatou set} $\mathcal F(R)$, which is defined to be the set of points $z \in \hat{\mathbb C}$ where the family $\{ R^n\, , \, n \geq 0 \}$ is normal in some neighbourhood of $z$, and its complement, the {\it Julia set} $\mathcal J(R)=\hat{\mathbb C} \setminus \mathcal F(R)$. The Fatou set is open and therefore $\mathcal J(R)$ is closed. Moreover, if the degree of the rational map $f$ is greater than or equal to 2, then the Julia set $\mathcal J(R)$ is not empty and is the closure of the set of repelling periodic points of $R$.

The connected components of $\mathcal F(R)$, called {\it Fatou components}, are mapped under $R$ among themselves. Sullivan (\cite{NoWanderingTheorem}) proved that any Fatou component of a rational map is either periodic or preperiodic. The Classification Theorem concludes that there are only four types of periodic Fatou components one of which consists of the connected components of the \textit{basin of attraction} of an attracting cycle. It is known that any attracting cycle of Fatou components contains  at least  one critical point. For a background on the dynamics of rational maps we refer to \cite{MilnorBook, BeardonBook,CarlesonGamelinBook}.

Accordingly, the key tool to understand the dynamical plane of $T_{p,\delta}$ is to control the dynamical behaviour of the critical orbits, i.e.\ orbits of critical points. For instance, if we assume that all roots of $p$ are simple we have 
\begin{equation}\label{eq:critics_Newton}
T^{\prime}_{p,0}(z)=N'_p(z)=\frac{p(z)p''(z)}{\left(p^{\prime}(z)\right)^2},
\end{equation}
and the critical points of Newton's method are the zeros of $p$ (which are superattracting fixed points of $T_{p,0}$) and  the zeros of $p^{\prime\prime}$. These latter critical points are usually called {\it free critical points} since they are not linked to any prescribed dynamics. Notice that the poles of $T_{p,0}^{\prime}$, that is, the zeros of $p^{\prime}$, are not critical points since it follows by   \eqref{eq:rat_T_lambda}     the map is one-to-one on a sufficiently small neighbourhood of each pole.

When considering $\delta\ne 0$ the degree of the  map changes  drastically  compared to $\delta=0$. Indeed, the degree increases from $d$ to $d^2$, and so the number of critical points increases from $2d-2$ ($d$ zeros of $p$ and $d-2$ zeros of $p''$) to $2d^2-2$. In some way, $T_{p,\delta}$, with $\delta\ne 0$, can be considered a {\it singular perturbation} of $T_{p,0}$, although the local perturbation takes place on the Julia set while in most cases where this theory applies the perturbation takes place on the Fatou set which allows a more rigid control of the global dynamics after perturbation (see for instance \cite{Trichotomy,Can1}). The following proposition gives a precise description of the critical points of $T_{p,\delta}$.

\begin{lemma}\label{lemma:critics}
Let $p$ be a monic polynomial of degree $d$ with all its roots being simple, and assume $\delta\in \mathbb C\setminus \{0\}$. Then, the critical points of $T_{p,\delta}$ can be classified as follows. 

\begin{enumerate}
\item[(a)] The zeros of $p$.  If $p(\alpha)=0$ then $\alpha$   is a critical point  with multiplicity 1 for $\delta \neq 1$ and multiplicity 2 for $\delta=1$ (Traub's method). 
\item[(b)] The zeros of $p'$ (which are poles of $T_{p,\delta}$).  If $p'(\beta)=0$, then $\beta$ is a critical point with multiplicity $d$.
\item[(c)] The zeros of $p''$. If $p''(\gamma)=0$ then $\gamma$ is a critical point and its multiplicity depends on higher derivates of $p$ at $\gamma$.
\item[(d)] Critical points that do not belong to any of the above cases. There are as many as 
\begin{itemize}
\item[(i)]  $d(d-1)$ if $\delta \neq 1$ and $\delta \neq d^d/(d-1)^{d-1}$.
\item[(ii)]  $d(d-2)$ for $\delta=1$ (Traub's method). 
\item[(iii)] $d(d-1)-2$ for $\delta = d^d/(d-1)^{d-1}$.
\end{itemize}
\end{enumerate}
\end{lemma}

\begin{proof}
The critical points of $T_{p,\delta}$ are given by the solutions of $T_{p,\delta}'(z)=0$ and eventually the poles of $T_{p,\delta}$. Using \eqref{eq:newton}, \eqref{eq:traub_delta} and \eqref{eq:critics_Newton} it is easy to see that 
\begin{equation}\label{eq:critical_traub}
T_{p,\delta}'(z)= \frac{p''(z)}{\left(p'(z)\right)^2} \left[ p(z)- \delta\;\frac{p'(N_p(z)) p(z)}{p'(z)}+ \delta p(N_p(z)) \right]. 
\end{equation}

We start by proving statement (a). Let $\alpha\in \mathbb C$ such that $p(\alpha)=0$ (notice that by assumption $p'(\alpha)\ne 0$). On the one hand, we have that $p(N_p(\alpha))=0$, and so substituting in \eqref{eq:critical_traub} it is clear that $T_{p,\delta}'(\alpha)=0$ and $\alpha$ is a critical point of $T_{p,\delta}$. On the other hand, we have that $N_p'(\alpha)=0$ and so doing some computations we get that 
$$
T_{p,\delta}''(\alpha)= N_p''(\alpha) (1- \delta)= \frac{p''(\alpha)}{p'(\alpha)} (1-\delta) \quad {\rm and} \quad T_{p,1}'''(\alpha) \neq 0,
$$ 
so statement (a) follows.

We turn now to statement (b).  It follows from \eqref{eq:rat_T_lambda} that the roots of $p'(z)$, i.e. poles of $T_{p,\delta}$, are preimages of $z=\infty$ of multiplicity $d+1$ and, hence, are critical points of $T_{p,\delta}$ of multiplicity  $d$ (one less than the order of the roots of $p'$ as poles of $T_{p,\delta}$). Statement (c) follows directly from \eqref{eq:critical_traub} since $p''(\gamma)=0$ implies $T_{p,\delta}^{\prime}(\gamma)=0$ and $T_{p,\delta}^{\prime\prime}(\gamma)=p^{\prime\prime\prime}(\gamma)/\left(p^{\prime}(\gamma)\right)^2$.

Finally, we prove statement (d). If $\delta\ne 1$ and $\delta \neq d^d/(d-1)^{d-1}$  we have already $d+d(d-1)+(d-2)=d^2+d -2$ critical points corresponding to zeros of $p$, $p'$ and $p''$. From Lemma \eqref{lem:infinity}(b) the global degree of the map is $d^2$ so the total number of  critical points is $2 d^2 -2$. Hence, we have  
$$
2d^2-2-(d^2+d -2)=d(d-1),
$$
extra critical points and thus (i) is proved. If $\delta=1$ the number of critical points corresponding to zeros of $p$, $p'$ and $p''$  is $2d+d(d-1)+(d-2)=d^2+2d -2$ while again the global degree of the map is $d^2$, so the total number of extra critical points is
$$
2d^2-2-(d^2+2d -2)=d(d-2).
$$
This proves (ii). Finally, if $\delta=d^d/(d-1)^{d-1}$ the number of  critical points related to the zeros of $p$, $p'$ and $p''$ is equal to  $d+d(d-1)+(d-2)=d^2+d -2$, while the global degree of the map is now $d^2-1$ (this is a degeneracy parameter where the degree decreases by 1), so the total number of extra critical points is
$$
2(d^2-1)-2-(d^2+d -2)=d(d-1)-2.
$$
This proves (iii).
\end{proof}

%
%
%
%
%--------------------QUADRATIC CASE 
%
%
%
%

\section{The quadratic case}\label{sec:quadratic}

In this section we assume that $p$ is a monic polynomial of degree 2. The main goal is to prove Theorem A with the assumption (a). We first study the case for which the quadratic polynomial has 2 different simple roots.  Let 
\begin{equation}\label{eq:p2}
p_2(z)=(z-\alpha_1)(z-\alpha_2),
\end{equation}
with $\alpha_1\ne \alpha_2 \in \mathbb C$ and $\alpha_1\neq \alpha_2$. A key feature to understand the dynamics of  $T_{p_2,\delta}$ is the fact that this map is conjugated to a map $G_{\delta}$ which depends on $\delta$ but does not depend on $\alpha_1$ and $\alpha_2$ (compare \cite{DampedTraubQuadratic}). Indeed, let 
\begin{equation}\label{eq:h}
h(z)= \frac{z-\alpha_2}{z-\alpha_1}
\end{equation}
be the M\"obius transformation which sends $\alpha_2$, $\alpha_1$, and $\infty$ to 0, $\infty$, and $1$, respectively. A simple computation shows that for all $z\in \hat{\mathbb C}$, we have   $G_{\delta}(z)\,=\,( h \circ T_{p_2,\delta} \circ h^{-1})(z)$, where
\begin{equation} \label{eq:g_delta}
G_\delta(z)= z^2 \ \frac{z^2+2z+(1-\delta)}{(1-\delta)z^2+2z+1}.
\end{equation}
In other words $T_{p_2,\delta}$  is analytically conjugate to the rational map $G_\delta$.
We would like to remark that a first study of the dynamics of the map $G_{\delta}$ was done in \cite{DampedTraubQuadratic}, where the authors studied for which parameters there are attracting fixed points  (other than the roots) and attracting cycles of period 2 and performed several numerical experiments.

The next lemma states that if $\delta\in \mathbb R$ then  $G_\delta$  is a Blaschke product (see \cite[p. 162-163]{MilnorBook} for details). This is an important property since Blaschke products leave the unit circle invariant.
\begin{lemma}\label{lemma:blaschke}
If  $\delta\in \mathbb R$ we have that $G_\delta$ is a  Blaschke product. 
\end{lemma}

\begin{proof}

The case $\delta=1$ follows directly by noticing that
$
G_1(z)= z^3 \frac{(z+2)}{(1+2z)}, 
$
is a Blaschke product.  Assume in what follows that $\delta \in\mathbb R $ and $\delta \neq 1$. 
To see that $G_{\delta}$ is a Blaschke product, we compute its zeros (other than $z=0$), denoted by $\xi_{\pm}:=\xi_{\pm}(\delta)$, and its poles, denoted by $w_{\pm}=w_{\pm}(\delta)$. They are given by
\begin{equation*}
\xi_{\pm}=-1 \pm \sqrt{\delta}  \quad {\rm and}\quad    w_{\pm} =  \frac{-1 \pm \sqrt{\delta}}{1-\delta} = \frac{\xi_{\pm}}{1-\delta}, \ \qquad {\rm if} \ \  \delta \ne 1;\\ 
\end{equation*}

Notice that  $\xi_+ \xi_-=1-\delta$,  $w_{-}=1/\xi_{+}$ and $w_{+}=1/\xi_{-}$. Consequently we can rewrite $G_\delta$, as
\[
G_{\delta}(z)= z^2 \frac{(z-\xi_+)(z-\xi_-)}{\xi_+ \xi_- (z-1/\xi_+)(z-1/\xi_-)}= z^2 \left(\frac{z-\xi_+}{1-z \, \xi_+}\right) \left(\frac{z-\xi_-}{1-z \, \xi_- }\right)
\]
\noindent proving thus that $G_{\delta}$ is a Blaschke product. \end{proof}
It follows from the previous lemma that, if $\delta\in\mathbb R$ then $G_{\delta}$ is symmetric with respect to the unit circle: so $G_\delta=\tau^{-1} \circ G_\delta \circ \tau$ where $\tau(z)=1/\overline{z}$. The next lemma states that the map $\varsigma(z)=1/z$ conjugates $G_{\delta}$ with itself for all $\delta\in\mathbb C$.  The proof is straightforward.

\begin{lemma} \label{lemma:conj}
Let $\delta\in\mathbb C$ and let $\varsigma(z)=1/z$. Then for all $z\in\hat{\mathbb C}$ we have $$G_\delta(z)=(\varsigma^{-1} \circ G_\delta \circ \varsigma)(z).$$ 
\end{lemma}

The following lemma deals with the critical orbits for $G_\delta$ and in particular shows the existence of a well defined one-dimensional $\delta$-parameter plane for the family $G_\delta$. The proof depends strongly on the previous lemma since the conjugacy $\varsigma$ allows us to tie the dynamics of the free critical orbits since both critical orbits behaves symmetrically.

\begin{lemma} \label{lemma:def_K}
Let $G_\delta$ be the family of maps given by \eqref{eq:g_delta} and assume that $\delta \neq 0$. The following statements hold.
\begin{itemize}
\item[(a)] The map   $G_\delta$ has degree 4, so it has 6 critical points counting multiplicity.
\item[(b)] If $\delta\ne 1$ then the critical points are given by 
$z=0, z=-1\ {\rm (double)},  z=\infty, \ {\rm and} \  z_{\pm}=c_{\pm}(\delta)$, where
\begin{equation}
\label{eq:critical points}
c_{\pm}(\delta)= \frac{-(2+\delta) \pm \sqrt{(2+\delta)^2 -4(1-\delta)^2}}{2(1-\delta)}.
\end{equation}
Moreover 
$$
\lim_{\delta\to 1} c_{+} (\delta) = 0 \quad {\rm and} \quad \lim_{\delta\to 1} c_{-} (\delta) = \infty. 
$$   
\item[(c)] If $\delta=1$ there are three double critical points given by $z=0, z=\infty$, and $z=-1$.  
\item[(d)] The orbit of all critical points different from $c_{\pm}$ is prescribed. Precisely $G_{\delta}(0)=0$, $G_{\delta}(\infty)=\infty$, $G_{\delta}(1)=1$ and $G_{\delta}(-1)=1$. In particular $z=0$ and $z=\infty$ are superattracting fixed points. We denote by $\mathcal A_{\delta}^{\star}(0)$ and $\mathcal A_{\delta}^{\star}(\infty)$ the immediate basins of attraction, respectively.
\item[(e)] If $\delta \in (0,1)$ we have $G_\delta(x)\ne x$ for all $x\in (0,1)$ and $G^{\prime}_\delta(x)>0$ for all $x\in (0,\infty)$. 
\item[(f)] The  critical points $c_{\pm}(\delta)$ satisfy $c_+(\delta)=1/c_-(\delta)$. Moreover, their orbits are symmetric with respect to $\varsigma(z)=1/z$, i.e.\ they satisfy $G_{\delta}^n(c_+(\delta))=1/G_{\delta}^n(c_-(\delta))$ for all $n\geq 1$.
\end{itemize}
In particular, $G_\delta$ defines a well defined one-dimensional $\delta$-parameter plane depending on the dynamical behaviour of the critical orbit $\{G_{\delta}^n \left( c_{+}(\delta)\right) \}_{n \geq 0}$ (compare Figure  \ref{fig:PP_Tdelta}).
\end{lemma}

\begin{proof}
Statements (a) to (e) can be obtained by simple computations. For statement (f), the fact that $c_+(\delta)=1/c_-(\delta)$ also follows from a simple computation. Using this, the fact that their orbits are symmetric with respect to $\varsigma(z)$ follows directly from Lemma~\ref{lemma:conj}.
\end{proof}

In Figure  \ref{fig:PP_Tdelta} we plot the $\delta-$plane of $G_{\delta}$ depending on the dynamical behaviour of the critical orbit $\{G_{\delta}^n \left( c_{+}(\delta)\right) \}_{n \geq 0}$.   We colour with a scaling from red (fast convergence) to blue (slow convergence) parameter values $\delta$ such that  the critical orbit is attracted by one of the two superattracting  fixed points located at the origin and infinity, while we colour in black $\delta-$values for which the critical orbit exhibits a different behaviour. In particular, the central red region $\mathcal K$ (a {\it hyperbolic component}) corresponds to parameter values where $c_{+}(\delta)$ belongs to the immediate basin of attraction of 0. More precisely, 

\begin{equation}\label{def:K}
\mathcal K=\{\delta \in \mathbb C \ | \ c_+(\delta)\in \mathcal A_\delta^{\star}(0)\}=\{\delta \in \mathbb C \ | \ c_-(\delta)\in \mathcal A_\delta^{\star}(\infty)\}.
\end{equation}

\begin{remark}\label{remark:1inK}
We have not proven that $\mathcal K$ is simply connected. For simplicity, from now on when we refer to $\mathcal K$  we restrict to the connected component of $\mathcal K$ which contains $\delta=1$. 
\end{remark}

\begin{figure}[hbt!]
\centering
    \includegraphics[width=250pt]{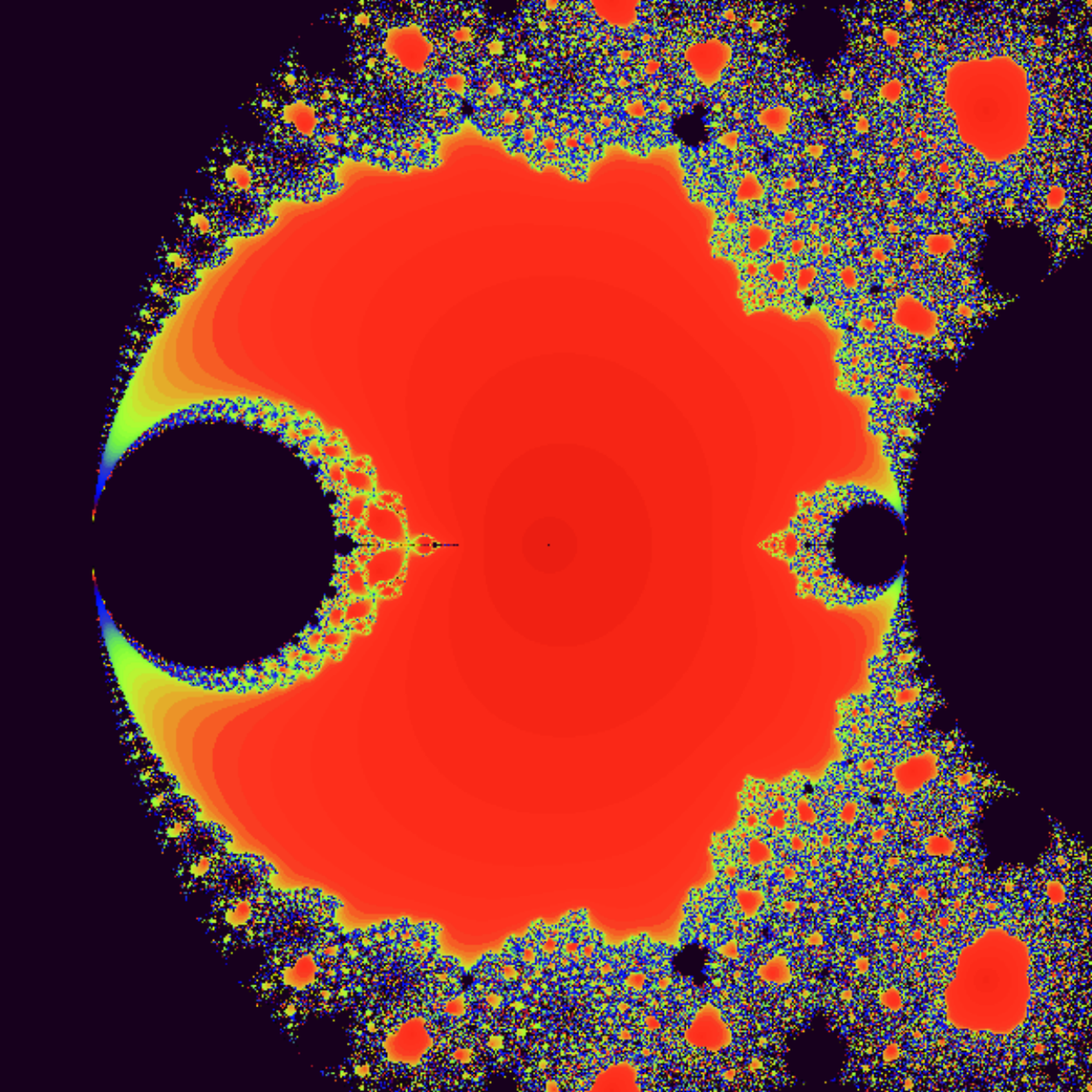} % Open a "figures" folder at the folder where the tex document is and put figures there.
    \put(-145,123){\scriptsize \textcolor{white}{$\bullet$}}
        \put(-145,116){\scriptsize $0$}
       \put(-126,123){\scriptsize \textcolor{white}{$\bullet$}}
        \put(-126,116){\scriptsize $1$}
    \put(-131,163){  $\mathcal K$}
     \caption{Parameter plane of  $G_{\delta}$ near $\delta=1$. The hyperbolic component $\mathcal K$ is the central red region and it is the hyperbolic component containing Traub's method ($\delta=1$). We also indicate the position of Newton's method ($\delta=0$) on the boundary of $\mathcal K$.} 
      \label{fig:PP_Tdelta}
\end{figure}
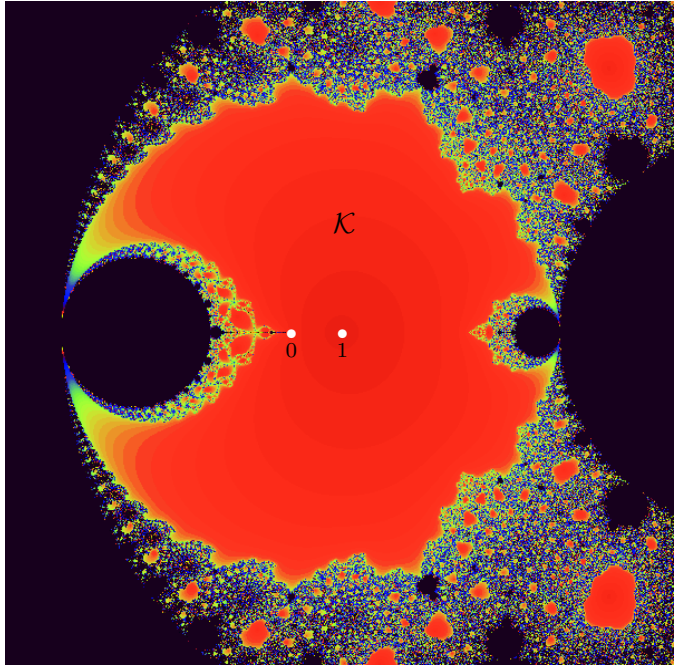

\begin{proof}[Proof of Theorem A(a).] First we consider the degenerate case $p(z)=(z-\alpha)^2$.  Easily we obtain  
\[
T_{p, \delta}(z) = N_{p}(z) -\delta \frac{p(N_{p}(z))}{p'(z)} = \frac{z}{2} + \frac{\alpha}{2}   - \delta \frac {(\frac{z}{2} + \frac{\alpha}{2} - \alpha)^2 }{2(z-\alpha)}  =  \frac{z}{2} + \frac{\alpha}{2}   - \delta   \frac{ z-\alpha}{8}. 
\]
Therefore, $T_{p,\delta}$ is a degree one map and,  if $| \delta - 4 |<8$, the point $z=\alpha$ is a global attracting fixed point; all points in $\mathbb C$ converge to $\alpha$ under iteration. This range of $\delta$'s includes $\delta=1$, so the statement follows.

Second we take $p=p_2$, a quadratic polynomial with two different roots $\alpha_1, \alpha_2\in \mathbb C$, $\alpha_1 \ne \alpha_2$. For simplicity in the exposition, we write $T_\delta:=T_{p_2,\delta}$ with $\delta \in [0,1]$.   It is well known that the proposition is true for $\delta=0$. Hence in what follows we take $\delta\in (0,1]$.

We have seen that  $T_\delta$ is conjugate to $G_\delta$ by the M\"obius map $h(z)$ \eqref{eq:h}, no matter the polynomial $p_2$ under consideration. Notice that since $h(\infty)=1$,  the unboundedness of the immediate attracting basins for $T_\delta$ is equivalent to show that $1\in \partial \mathcal A_\delta^{\star}(0) \cap \partial  \mathcal A_\delta^{\star}(\infty)$ for $G_\delta$. This follows directly from Lemma \ref{lemma:def_K}(d-e). 

It remains to prove that $\mathcal A_\delta^{\star}(\alpha_j),\ j=1,2$, are simply connected. To see this we argue as follows. First we prove that  $(0,1]\subset \mathcal K$ (notice that Remark \ref{remark:1inK} implies $1\in \mathcal K$). Second we prove that for $\delta=1$ the result is true (notice this is enough for the conjecture). Finally, the result follows since simple connectivity is preserved for all parameters in the same hyperbolic component (see Figure \ref{fig:PP_Tdelta}).

 A simple computation shows that for $\delta\in (0,1)$ we have
$$
c_{-}(\delta)<-1<c_{+}(\delta)<0 \quad {\rm and} \quad G_\delta(c_{+}(\delta))=-\left(c_{+}(\delta)\right)^3\in (0,1).
$$
Notice that $G_\delta$ has two vertical asymptotes but they are located to the left of $c_{+}(\delta)$. In particular, $c_{+}(\delta)\in \mathcal A_\delta^{\star}(0)$ (and $c_{-}(\delta)\in \mathcal A_\delta^{\star}(\infty)$ by the symmetry) and so $(0,1]\subset \mathcal K$. 

Assume now that $\delta=1$. From Lemma \ref{lemma:def_K}(a,c), neither  $ \mathcal A_1^{\star}(0)$ nor  $\mathcal A_1^{\star}(\infty)$ contain extra critical points different from $z=0$ and $z=\infty$. Hence the local B\"otcher coordinates defined in a sufficiently small neighbourhood of $z=0$ and $z=\infty$  (compare \cite{MilnorBook})   extend to the whole immediate basin of attraction, which implies that they both are simply connected.  

To finish the proof we observe that simple connectivity is preserved inside the  hyperbolic component $\mathcal K$. For instance, we can use the fact that inside a hyperbolic component Julia sets are quasi-conformally conjugated, or $\mathcal J$-stable, see \cite[Section 4.1]{McMullenBook}.  Since $1\in \mathcal K$ the result follows.

\end{proof}

%
%
%
%
%---------------------------------------Z^n - 1 CASE -------------------------
%
%
%
%

\section{The case $z^n-\beta$}\label{sec:pol}

In this section we consider a family of higher degree polynomials. More precisely,   
$$
p_{n,\beta}(z)=z^n-\beta, 
$$
where $n \geq 3$ and $\beta\in\mathbb C$.  We firstly consider the (degenerate) case $\beta =0$. Easy computations show that the  damped Traub's map applied to the polynomial $p_{n,0}(z)=z^n$ is given by  $$T_{p_n,0,\delta}(z)=\left(\frac{n-1}{n}\right) \left(1- \delta \frac{(n-1)^{n-1} }{n^n} \right) z $$.

Thus, $T_{p_n,0,\delta}$ is a degree 1 map and for $\delta \in [0,1]$ we have that $\mathcal A^*(0)=\mathbb C$ since the origin is an attracting fixed point. 

Hereafter, we take $\beta\neq 0$.  Using the symmetries of the family (to keep the dynamical behaviour of the critical points under control) and introducing some new tools, in this section we prove Theorem A(b) as a direct consequence of Proposition \ref {prop:simplyconN} (simple connectivity) and Proposition \ref{prop:unboundedN} (unboundedness).

  \begin{figure}[htb]
  	\centering
  	\subfigure[\scriptsize{$n=2$ }]{\begin{tikzpicture}
  			\begin{axis}[width=0.4\textwidth, axis equal image, scale only axis,  enlargelimits=false, axis on top]
  				\addplot graphics[xmin=-2,xmax=2,ymin=-2,ymax=2] {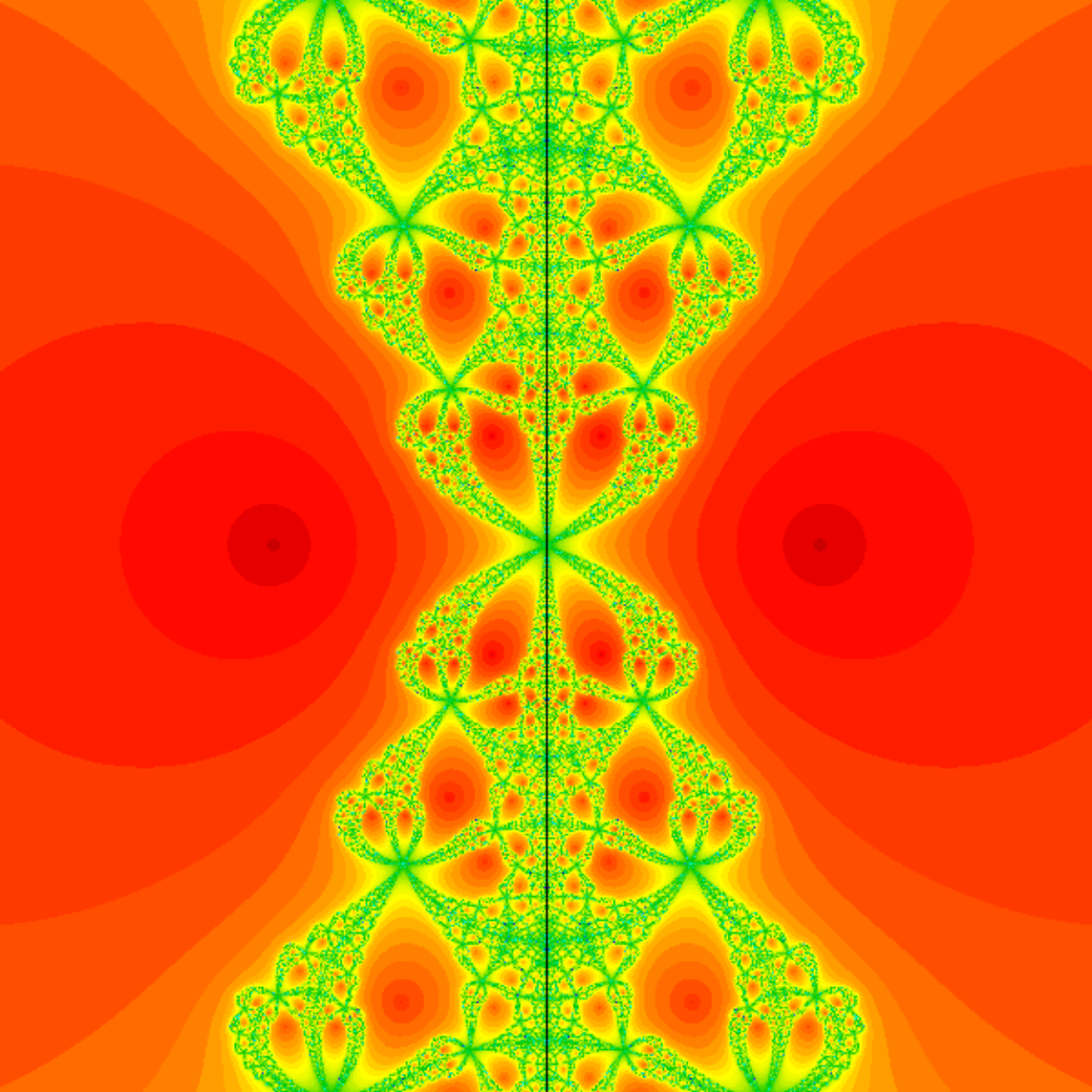};
  			\end{axis}
  	\end{tikzpicture}}
  	\subfigure[\scriptsize{$n=3$ }  ]{	\begin{tikzpicture}
  			\begin{axis}[width=0.4\textwidth, axis equal image, scale only axis,  enlargelimits=false, axis on top]
  				\addplot graphics[xmin=-2,xmax=2,ymin=-2,ymax=2] {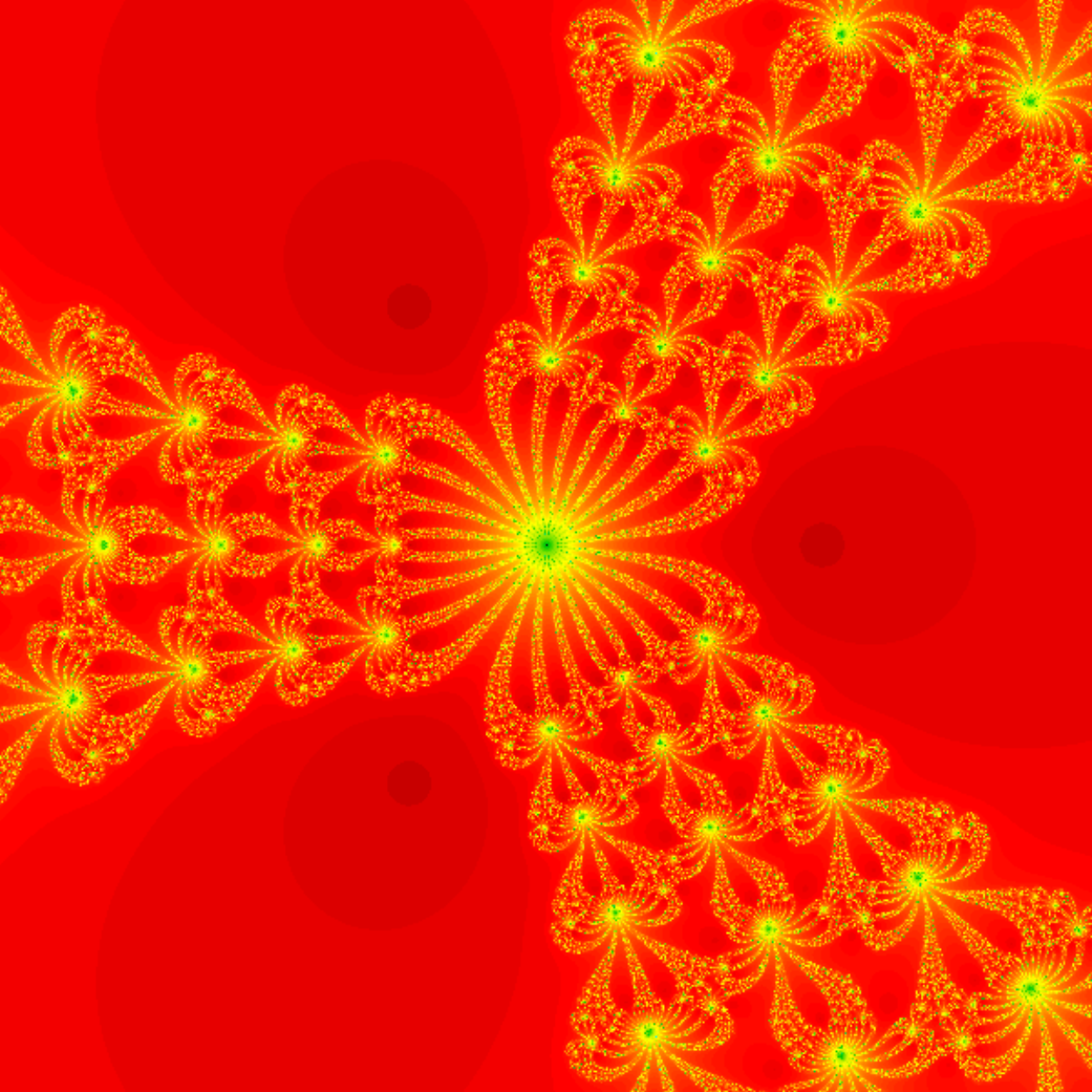};
  			\end{axis}
  	\end{tikzpicture}}
  	\subfigure[\scriptsize{$n=4$}  ]{	\begin{tikzpicture}
  			\begin{axis}[width=0.4\textwidth, axis equal image, scale only axis,  enlargelimits=false, axis on top]
  				\addplot graphics[xmin=-2,xmax=2,ymin=-2,ymax=2] {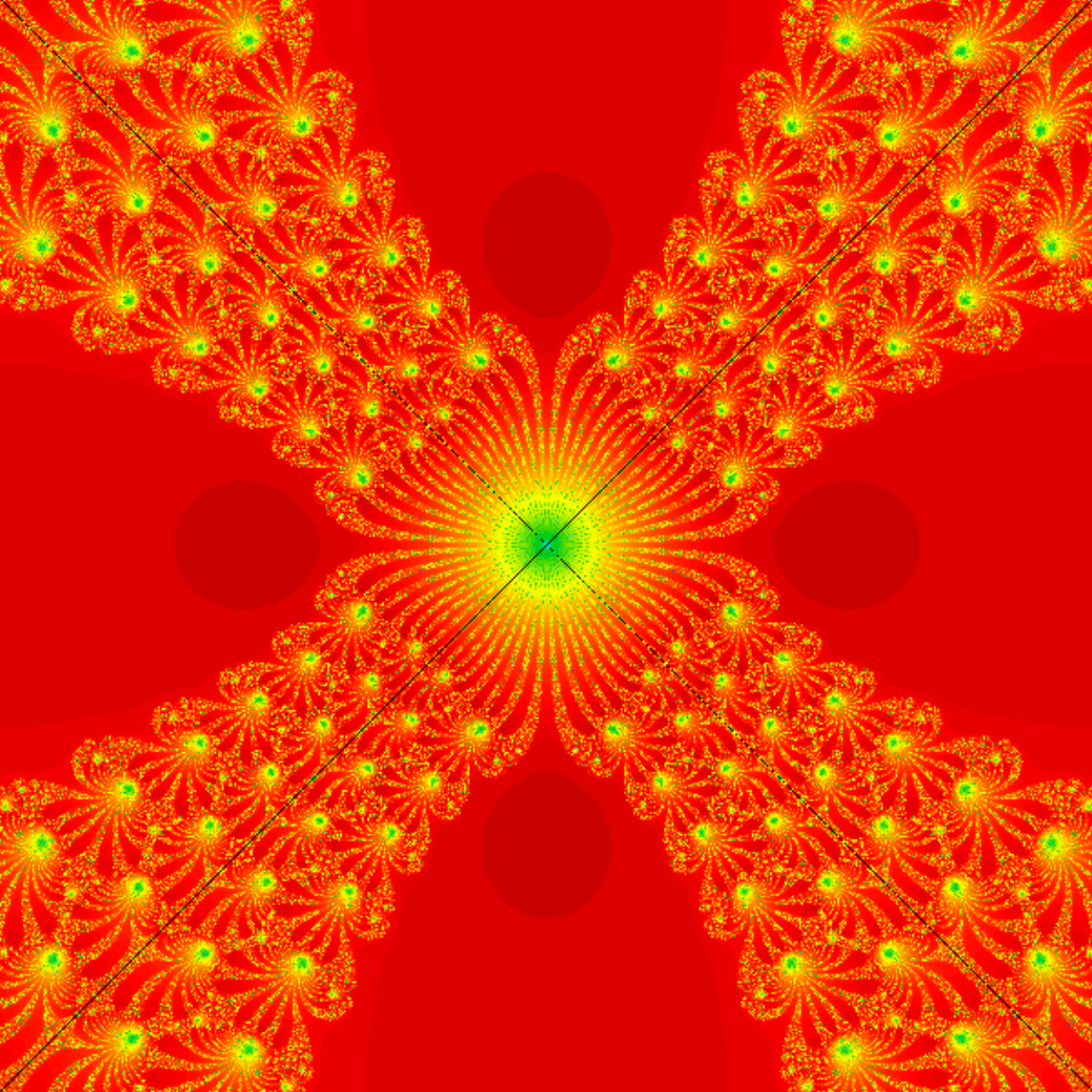};
  			\end{axis}
  	\end{tikzpicture}}
  	\subfigure[\scriptsize{$n=5$}  ]{ \begin{tikzpicture}
  			\begin{axis}[width=0.4\textwidth, axis equal image, scale only axis,  enlargelimits=false, axis on top]
  				\addplot graphics[xmin=-2,xmax=2,ymin=-2,ymax=2] {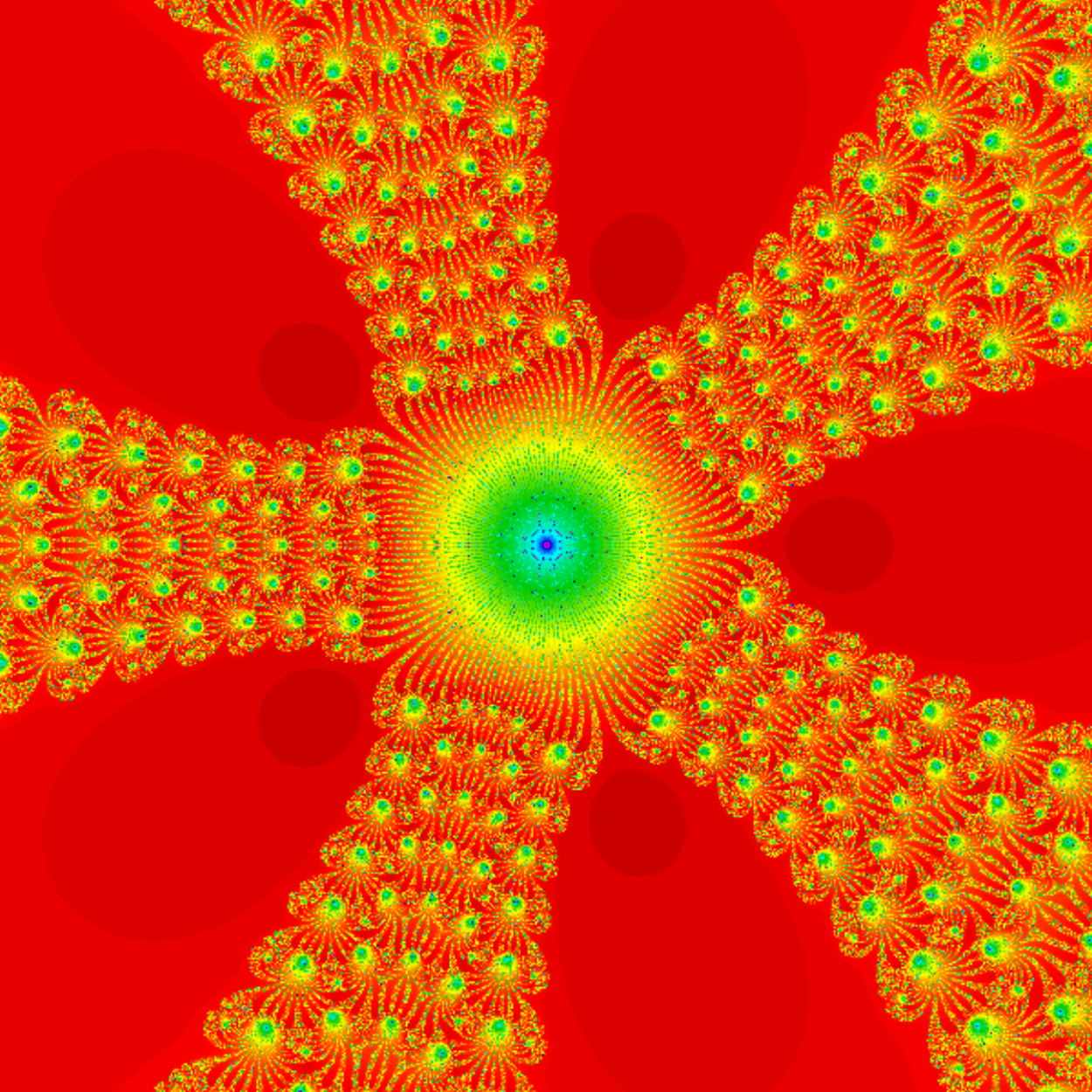};
  			\end{axis}
  	\end{tikzpicture}}
  	\caption{\small{Dynamical planes of Traub's method applied to the polynomial $p_n=z^n-1$.}}
  	\label{fig:TraubN}
  \end{figure}

 Let
\begin{equation}
\label{eq:graunaux}
N_{p_{n,\beta}}(z)=\frac{(n-1)z^n+\beta}{nz^{n-1}} \quad {\rm and} \quad {p_{n,\beta}}\left(N_{p_{n,\beta}}(z)\right)=\frac{\left((n-1)z^n+\beta\right)^n-\beta\  n^nz^{n(n-1)}}{n^nz^{n(n-1)}}.
\end{equation}
Then, the rational map obtained when applying the family of numerical methods to the polynomials $p_{n,\beta}$ is 
\begin{equation}\label{eq:graun}
	\begin{split}
	T_{p_{n,\beta},\delta}(z)&=N_{p_{n,\beta}}(z)-\delta\ \frac{{p_n}\left(N_{p_{n,\beta}}(z)\right)}{{p^{\prime}_{n,\beta}}(z)}\\
	&=\frac{n^{n}(n-1)z^{n^2}+\beta\ n^{n}(1+\delta)z^{n(n-1)}-\delta((n-1)z^n+\beta)^n}{n^{n+1}z^{n^2-1}}.
	\end{split}
\end{equation}

Notice that the $n$th-roots of $\beta$ are always superattracting fixed points of $T_{p_{n,\beta},\delta}$ since they are simple zeros of $p_{n,\beta}$. 

In order to understand the dynamics of $T_{p_{n,\beta},\delta}$ it is enough to work with $T_{p_{n},\delta}$, where 
$$p_n(z):=p_{n,1}(z)=z^n-1.$$ 
Indeed,  $T_{p_n,\delta}$ and $T_{p_{n,\beta},\delta}$ are conjugate for every $\beta\in\mathbb{C}\setminus\{0\}$. This is the content of the following lemma. The proof is straightforward and so omitted.

\begin{lemma}\label{lemma:conjugacioN}
Let $\beta\in\mathbb{C}\setminus\{0\}$ and let $\eta(z)=z\sqrt[n]{\left(1/\beta\right)} $. Then, for all $z\in \hat{\mathbb{C}}$ we have $T_{p_n,\delta}(z)=(\eta^{-1}\circ T_{p_{n,\beta},\delta}\circ \eta)(z)$
\end{lemma}
 
Throughout the section we work with the maps $T_{p_{n},\delta}$ since, by Lemma~\ref{lemma:conjugacioN}, every result proved for $T_{p_{n},\delta}$ can be generalized to $T_{p_{n,\beta},\delta}$.
In Figure~\ref{fig:TraubN} we can observe the dynamical planes obtained when applying Traub's method to $p_n$ for several values of $n$. 

The next lemma states that the maps $T_{p_n,\delta}$ are symmetric with respect to rotation by an $n$th root of unity. Its proof is straightforward and so omitted.
\begin{lemma}\label{lem:symmetry}
	Let  $\phi(z)=\xi z$ with $\xi^n=1$. Then $T_{p_n,\delta}(\phi(z))=\phi(T_{p_n,\delta}(z))$.
\end{lemma}

This property is relevant since it ties the orbit of every critical point $c$ different from $z=0$ (or $z=\infty$) to the orbit of the critical points $\xi^i c$, $i=1,...,n-1$. This fact decreases drastically the degree of freedom of the family $T_{p_n,\delta}$ of degree $n^n$ rational maps. Indeed, from Lemma~\ref{lemma:critics} we know that, if $\delta\neq 0$ and $ \delta \neq n^n/(n-1)^{n-1}$, the critical points of $T_{p_n,\delta}$ are the $n$th-roots of  unity (which correspond to the zeros of $p_n$ and, hence, are superattracting), the point $z=0$ (which is the only zero of $p_n'$ and $p_n''$) and $n(n-1)$ other critical points. Since $z=0$ is a pole and $z=\infty$ is a fixed point for  $ \delta \neq n^n/(n-1)^{n-1}$ (see Lemma~\ref{lem:infinity}), the maps $T_{p_n,\delta}$ have $n(n-1)$ free critical points. However, by Lemma~\ref{lem:symmetry} these $n(n-1)$ critical points can be grouped in sets of $n$ critical points which have symmetrical orbits. Therefore, the maps $T_{p_n,\delta}$ only have $n-1$ free critical orbits. 
 
We want to remark that the dynamics near $\infty$ in the case $ \delta =n^n/(n-1)^{n-1}$ is very different  from that when  $\delta \neq n^n/(n-1)^{n-1}$. Indeed, for  $ \delta =n^n/(n-1)^{n-1}$ we have $T_{p_n,\delta}(\infty)=0$, so the points $z=0$ and $z=\infty$ form a superattracting cycle of period 2. However,  for  $ \delta \neq n^n/(n-1)^{n-1}$ the point $z=\infty$ is fixed and, hence, $z=0$ is prefixed.
 
% The following proposition tells us that the immediate basins of attraction of the $n$th-roots of  unity under $T_{p_n,\delta}$ are simply connected. 
  
 \begin{prop}\label{prop:simplyconN}
The immediate basins of the $n$th-roots of  unity under $T_{p_n,\delta}$ are simply connected. 
 \end{prop}
\begin{proof}
	It follows from Lemma~\ref{lem:symmetry} that the immediate basins of attraction of the $n$th-roots of  unity are symmetric with respect to rotation by an $n$th-root of unity. Therefore, either all immediate basins of attraction are simply connected or they all are multiply connected. 

Assume that they are multiply connected.  It follows from the maximum modulus principle that all immediate basins of attraction have to surround a pole. Since $z=0$ is the only pole of $T_{p_n,\delta}$, all immediate basins of attraction of the $n$th-roots of unity surround $z=0$. Since they are symmetric with respect to rotation by an $n$th-root of unity, that would imply that the immediate basins of attraction have non-empty intersection, which is a contradiction.
\end{proof}

We now turn our attention to the unboundedness. We already know from Lemma~\ref{lem:infinity} that for some $\delta$-parameters the point $z=\infty$ is an attracting fixed point of $T_{p_n,\delta}$ and, hence, the immediate basins of attraction of the roots of $p_n$ cannot be unbounded for such parameters. Notice that the parameter $\delta=1$ does not belong to this set of \textit{bad} parameters, so the unboundedness part of the Conjecture still makes sense. The next result states that the immediate basins of attraction of the roots are unbounded if $\delta\in[0,1]$.

\begin{prop}\label{prop:unboundedN}
Let  $\delta\in[0,1]$ and $n\geq 3$. Then, the immediate basins of attraction of the $n$th-roots of unity under the map $T_{p_n,\delta}$ are unbounded. 
\end{prop}

\begin{proof}
The case $\delta=0$ is well known, see \cite{ConnectivityJulia}. So we may assume $\delta\in(0,1]$.  The maps $T_{p_n,\delta}$ are symmetrc with respect to rotation by an $n$th-root of  unity (Lemma~\ref{lem:symmetry}). 
Since $\delta$ is real, the real line is forward invariant under $T_{p_n,\delta}$. Moreover $T_{p_n,\delta}(1)=1$ and $T_{p_n,\delta}^{\prime}(1)=0$. Hence, $x=1$ is a superattracting fixed point for $T_{p_n,\delta}$. Denote by $\mathcal{A}^\star(1)$ the immediate basin of attraction of $x=1$.  If we prove that for all $x>1$ we have $1<T_{p,\delta}(x)<x$, we can conclude that $[1,\infty) \subset \mathcal{A}^{\star}(1)$ and, hence, the result follows. 

The inequality $T_{p_n,\delta}(x)< x$  is equivalent to 
\begin{equation}\label{eq:smallthanx}
\frac{p_n(x) + \delta  p_n\left(N_{p_n}(x)\right)}{p_n'(x)} >0.
\end{equation}
Since $p_n(x)=x^n-1$ and $p_n'(x)=n x^{n-1}$, we have $p_n(x)>0$ and $p_n'(x)>0$ for all $x>1$. Recall that 
\[
N_{p_n}(x)= \frac{(n-1)x^n+1}{nx^{n-1}}.
\]
Therefore, $N_{p_n}(1)=1$ and we know that $N_{p_n}(x)>1$ for all $x >1$. Thus we can conclude that $p_n(N_{p_n}(x))>0$ for all $x >1$. This implies that  the inequality \eqref{eq:smallthanx} is satisfied for $x>1$.

Now we prove $T_{p_n,\delta}(x)>1$ for all $x>1$. Easy manipulations imply that the equation $T_{p_n,\delta}(x)=1$ (for $x>1$) can be written  as
\begin{equation}\label{eq:pn_1}
p_n\left( N_{p_n}(x) \right)=\frac{1}{\delta}\left((n-1)x^{n}-nx^{n-1}+1\right).
\end{equation} 
Since 
$$
p_n\left(N_{p_n}(x)\right)=\frac{1}{n^nx^{n^2-n}}\left[\left((n-1)x^n+1\right)^n-n^nx^{n^2-n}\right]
$$
we have that \eqref{eq:pn_1} is equivalent to
\begin{equation}\label{eq:pn_2}
\left((n-1)x^n+1\right)^n=\frac{n^n}{\delta}\left[(n-1)x^{n^2}-nx^{n^2-1}+(1+\delta)x^{n^2-n}\right].
\end{equation}
The left hand side of this equation expands as 
\begin{equation} \label{eq:Newton}
\left((n-1)x^n+1\right)^n=(n-1)^nx^{n^2}+n(n-1)^{n-1}x^{n^2-n}+A_n(x) 
\end{equation}
where 
$$
A_n(x)=\sum_{j=2}^n 
\left(\begin{array}{c}n \\
j\end{array}\right)
(n-1)^{n-j}x^{n^2-nj}.
$$
Thus, if we set 
$$
S_{\delta,n}(x):=(n-1)\left(\frac{n^n}{\delta}-(n-1)^{n-1}\right)x^{n^2}-\frac{n^{n+1}}{\delta}x^{n^2-1}+\left(\frac{n^n(1+\delta)}{\delta}-n(n-1)^{n-1}\right)x^{n^2-n}-A_n(x)
$$
 equation \eqref{eq:pn_2} rewrites as 
\begin{equation}\label{eq:pn_3}
S_{\delta,n}(x)=0.
\end{equation}
We claim that $S_{\delta,n}(x)=0$ either  has a unique (triple) positive root at $x=1$ (case $\delta=1$), or two roots (case $\delta\in(0,1)$): $x=x_0<1$ (simple) and $x=1$ (double). Consequently the equation $T_{p_n,\delta}(x)=1$ has no solutions for $x>1$. Moreover, since 
$$
\lim_{x\to \infty} T_{p_n,\delta}=\infty,
$$
we conclude that $T_{p_n,\delta}(x)>1$ for all $x>1$, as desired.

To justify the claim we notice that using  \eqref{eq:Newton} we can compute the following expressions 
\begin{equation}\label{eq:an1}
\begin{split}
&A_n(1)=n^n-(n-1)^{n-1}(2n-1),\\
&A^{\prime}_n(1)=(n-1)\left[n^{n+1}-2n^2(n-1)^{n-1}\right] \ \ {\rm and} \\
&A^{\prime\prime}_n(1)=n^2(n-1)^2\left[n^n-(n-1)^{n-2}(2n^2-n-2)\right]
\end{split}
\end{equation}

Then, from the above expression of $S_{\delta,n}$ and \eqref{eq:an1} we get
\begin{equation}
\begin{split}
&S_{\delta,n}(1)=n^n-(n-1)^{n-1}(2n-1)-A_n(1)=0,\\
&S_{\delta,n}'(1)=(n-1)\left[n^{n+1}-2n^2(n-1)^{n-1}\right]-A^{\prime}_n(1)=0 \ \ {\rm and} \\
&S_{\delta,n}''(1)=\frac{n^{n+1}(n-1)}{\delta}\left(\delta n^2-\delta n  +1 -\delta \right)-n^2(n-1)^n(2n^2-n-2)
\end{split}
\end{equation}
Some easy computations show that $S_{1,n}''(1)=0$ and $S_{\delta,n}''(1)>0$ for all $\delta\in (0,1)$. 

Applying Descarte's rule, the polynomial equation $S_{\delta,n}(x)=0$ has either 1 or 3 positive real solution(s), counting multiplicity (this is immediate since the consecutive coefficients change sign three times). Noticing that $S_{\delta,n}(0)=-1$ we conclude the claim.
\end{proof}

We can now prove Theorem A(b).
\begin{proof}[Proof of Theorem A(b)] By Propositions~\ref{prop:simplyconN} and \ref{prop:unboundedN} we know that the immediate basins of attraction of $n$th roots of  unity under $T_{p_n,\delta}$ are simply connected and unbounded. The immediate basins of attraction of the $n$th roots of $\beta$ under $T_{p_{n,\beta},\delta}$ are also simply connected and unbounded since the conjugacy $\eta(z)=\sqrt[n]{\beta}z$ (see Lemma~\ref{lemma:conjugacioN}) sends the basins of attraction of the $n$th roots of unity to the basins of attraction of the $n$th roots of $\beta$.

\end{proof}

\subsection{The cubic case}
In the study of the maps $T_{p_n,\delta}$ we have analysed the topological properties of the immediate basins of attraction of the $n$th-roots of unity but we have not provided any control on the dynamics of the orbits of the free critical points. We finish this section studying the dynamics of the critical orbits for the cubic case, i.e.\ we study the rational maps  $T_{p_3,\delta}$ obtained when applying the numerical methods to $p_3(z)=z^3-1$. Their formula is given by 
\begin{equation*}
	T_{p_3,\delta}(z)=\frac{(54-8\delta)z^{9}+(27+15\delta)z^{6}-6\delta z^3-\delta}{81z^{8}}.
\end{equation*}
Their derivative is given by
$$T'_{p_3,\delta}(z)=\frac{(54-8\delta)z^{9}-(54+30\delta)z^{6}+30\delta z^3+8\delta}{81z^{9}}.$$
The critical points of $T_{p_3,\delta}$ are $z=0$ (which is a preimage of $z=\infty$ of multiplicity 8) and the zeros of $T'_{p_3,\delta}(z)$. These latter critical points correspond to the third roots of  unity (which are superattracting fixed points) and the points 
\begin{equation}\label{eq:critcubic}
	c_{\xi, \pm}=\xi\sqrt[3]{r_{\pm}}, \quad \mbox{where} \quad r_{\pm}=\frac{19\delta\pm \sqrt{27(16\delta+11\delta^2)}}{54-8\delta}
\end{equation}
and $\xi^3=1$. The next proposition describes where these free critical points lie when $\delta\in(0,1]$. Recall that, given an $n$th-root of  unity $\xi$, we denote by $\mathcal{A}^*(\xi)$ the immediate basin of attraction of $\xi$ under $T_{p_3,\delta}$.

\begin{prop}\label{prop:cubic}
	Let $\delta\in(0,1]$. Then, for any third root of  unity $\xi$ we have $c_{\xi, +}\in \mathcal{A}^*(\xi)$ and  $T_{p_3,\delta}(c_{\xi, -})\in \mathcal{A}^*(\xi)$. In particular, the set of parameters $(0,1]$ belongs to a hyperbolic component for which all free critical points lie in the basins of attraction of the third roots of  unity.
\end{prop}

\begin{proof}
	Fix $\delta\in(0,1]$. By the symmetry with respect to rotation by third roots of  unity (see Lemma~\ref{lem:symmetry}), it suffices to show  that $c_{1, +}\in \mathcal{A}^*(1)$ and  $T_{p_3,\delta}(c_{1, -})\in \mathcal{A}^*(1)$. Notice that both $c_{1, +}$ and $c_{1, -}$ are real since $\delta$ is real (see \eqref{eq:critcubic}). Therefore, it is enough to restrict to the real dynamics of $T_{p_3,\delta}$. Since $\delta\in(0,1]$ we have $c_{1, +}>0$. Since $\lim_{x\rightarrow -\infty} T_{p_3,\delta}(x)=-\infty$ and $\lim_{x\rightarrow 0^-} T_{p_3,\delta}(x)=-\infty$ for $\delta\in(0,1]$, it follows that $c_{1, -}$ is negative and is a local maximum (see Figure \ref{fig:cubicreal}). Moreover, the global maximum of $T_{p_3,\delta}$ among negative real numbers is $T_{p_3,\delta}(c_{1, -})$.
	
	\begin{figure}
		\centering
		\includegraphics[width=10cm]{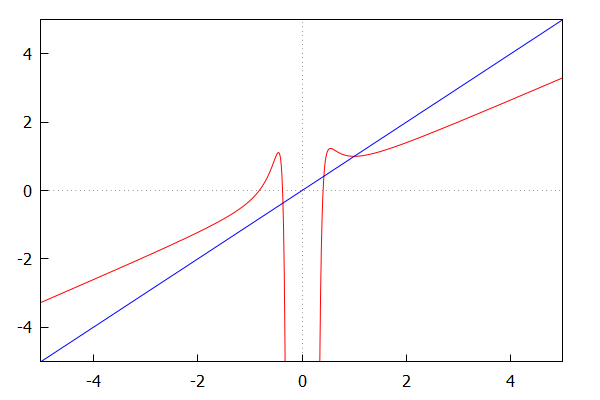}
		\caption{Graphic of $T_{p_3,\delta}(x)$ for $\delta=0.1$. We also draw the line $y=x$.}
		\label{fig:cubicreal}
	\end{figure}
	
	We first prove that $c_{1, +}\in \mathcal{A}^*(1)$. If $\delta=1$ (which corresponds to Traub's Method) then $c_{1, +}= 1$ (see \eqref{eq:critcubic}) and we are done. If $\delta\in (0,1)$ then the superattracting fixed point $x=1$ is a local minimum (see Figure \ref{fig:cubicreal}). Indeed, the second derivative of $T_{p_3,\delta}$ is given by
	$$T''_{p_3,\delta}(x)=\frac{(18+10\delta)x^{6}-20\delta x^3-8\delta}{9x^{10}}$$
	and $T''_{p_3,\delta}(1)=2-2\delta$, which is negative for $\delta\in (0,1)$. Since $\lim_{x\rightarrow 0^+} T_{p_3,\delta}(x)=-\infty$ and $c_{1, +}$ is the only real and positive critical point other than $x=1$ if $\delta\in (0,1)$, it follows that  $c_{1, +}\in(0,1)$ and $T_{p_3,\delta}$ has a local maximum at $c_{1, +}$. We can conclude that the segment $[c_{1, +},1]$ is mapped into the segment $[1,+\infty)$ under $T_{p_3,\delta}$. By the proof of Proposition~\ref{prop:unboundedN} we know that $[1,+\infty)\subset \mathcal{A}^*(1)$. Therefore, we have $c_{1, +}\in \mathcal{A}^*(1)$.
	
	To finish the proof we have to show that  $T_{p_3,\delta}(c_{1, -})\in \mathcal{A}^*(1)$. First of all observe that $c_{1, -}\notin \mathcal{A}^*(1)$. Otherwise, it would follow from the Schwartz Reflection Principle that $ \mathcal{A}^*(1)$ surrounds the pole $z=0$, which is impossible since  $\mathcal{A}^*(1)$ is simply connected (see Proposition~\ref{prop:simplyconN}). Since $T_{p_3,\delta}(-1/2)=1$ and  $c_{1, -}$ is the global maximum of $T_{p_3,\delta}$ among negative real numbers for $\delta\in(0,1]$, it follows that $T_{p_3,\delta}(c_{1, -})\in [1,+\infty)$. The result holds since $[1,+\infty)\subset \mathcal{A}^*(1)$ (see the proof of Proposition~\ref{prop:unboundedN}).

\end{proof}

\begin{figure}[hbt!]
	\centering
	\begin{tikzpicture}
		\begin{axis}[width=250pt, axis equal image, scale only axis,  enlargelimits=false, axis on top]
			\addplot graphics[xmin=-4,xmax=6,ymin=-5,ymax=5] {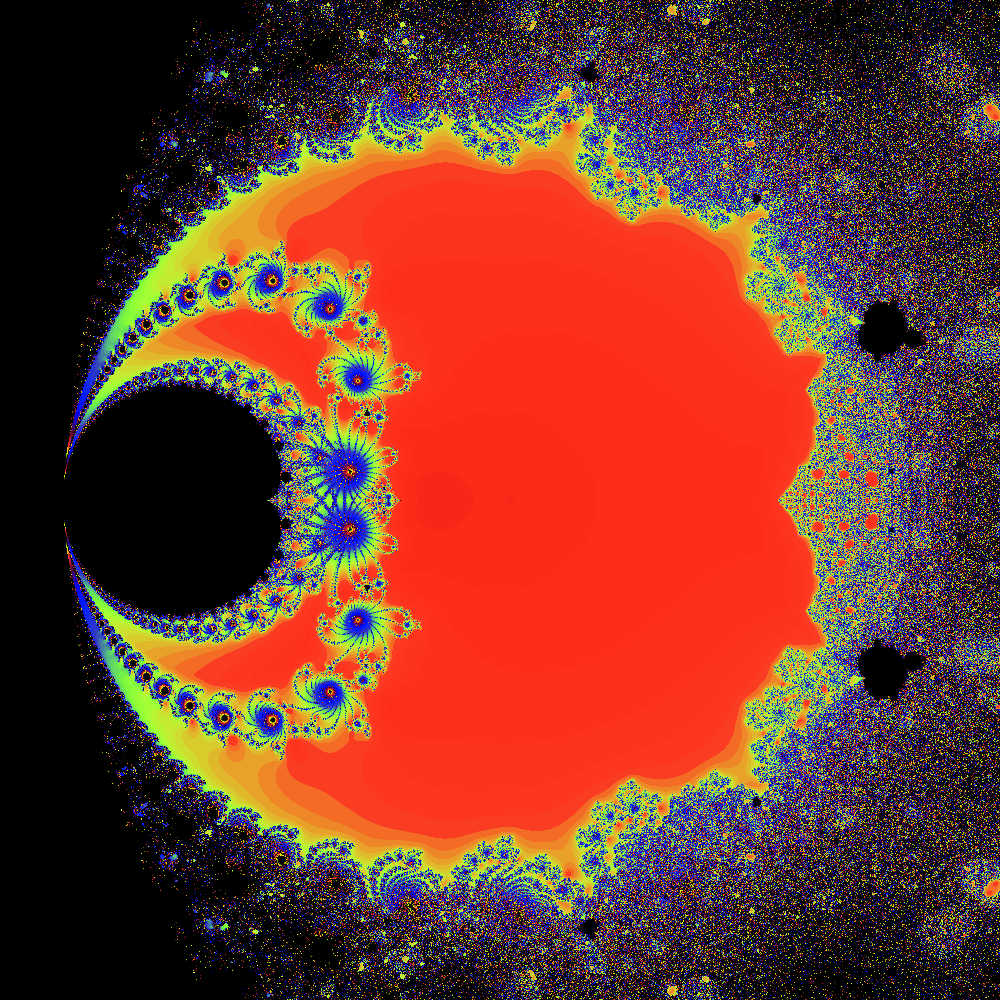};
		\end{axis}
	\end{tikzpicture}
	\put(-143,120){ $\times 0$}
	\put(-120,120){$\times 1$}
	\caption{Parameter plane of damped Traub's method applied to $p_3$. } 
	\label{fig:param_cubic}
\end{figure}
In Figure~\ref{fig:param_cubic} we can observe the parameter plane of damped Traub's method applied to $p_3$. Notice that there are 2 different critical orbits ({\it modulo} symmetries). This figure is produced by iterating the critical points $c_{1,+}$ and $c_{1,-}$. If any of the critical points does not converge to a root we plot the parameter in black. Otherwise we use a scaling of colours which indicates the speed of convergence to a root of the critical point which takes longer time to converge. The largest coloured region corresponds to the hyperbolic component which contains the segment of parameters $(0,1]$.
%
%
%
%
%----------------------NUMERICAL EVIDENCES-------------
%
%
%
%

\section{Numerical evidence and conclusions}\label{sec:numconcl}

Up  to now we have studied the Conjecture of simple connectivity and unboundedness of the immediate basins of attraction of the roots of polynomials under Traub's method from an analytical point of view. In this section we present a further discussion together with numerical evidence to justify why we think that the Conjecture is true and how damped Traub's family can help us to prove the conjecture in full generality.

As we have mentioned before, for $\delta$ close enough  to $0$ we can formulate damped Traub's method $T_{p,\delta}$ as a singular perturbation of Newton's method $N_p$. However, compared to usual singular perturbations where the poles are added to superattracting fixed cycles (see for instance \cite{Trichotomy}), this singular perturbation is done by adding extra preimages of $\infty$ to the zeros of $p'(z)$, which are already poles of $N_p$. Since $z=\infty$ is a repelling fixed point for Newton's method, it follows that this singular perturbation is done over the Julia set, which makes it much more difficult to control.

On the other hand, if a point belongs to the basin of attraction of a root under $N_p$, then it also belongs to the basin of attraction of the root under $T_{p,\delta}$, if $|\delta|$ is small enough. The unboundedness of the basin of attraction of a root of the polynomial $p$ under $T_{p,\delta}$ can be inherited from that of  $N_p$  at least for $|\delta|$  small.  In fact, it can be proven that if $|\delta|$ is small enough, then the number of \textit{accesses} to $\infty$ from the basin  of a root $\alpha$ under  $T_{p,\delta}$ is  at least  equal to the number of accesses to $\infty$ under the $N_p$ (see Figure~\ref{fig:numericdinam} (a-b)). Therefore, if $T_{p,1}$ happens to have a bounded immediate basin of attraction associated to a root of $p$ (i.e., all accesses to infinity inside the basin of the root have been truncated) there should be a bifurcation, which requires the \textit{intervention} of critical points. However, all simulations seem to indicate that the critical points that appear after perturbation cannot be responsible for that.

\begin{comment}
\begin{figure}[htb]
	\centering
	\subfigure[Newton's method.]{\begin{tikzpicture}
			\begin{axis}[width=0.4\textwidth, axis equal image, scale only axis,  enlargelimits=false, axis on top]
				\addplot graphics[xmin=-0.2,xmax=1.2,ymin=0,ymax=1.4] {./figures/newton.pdf};
			\end{axis}
	\end{tikzpicture}}
	\subfigure[Traub's method.  ]{\begin{tikzpicture}
			\begin{axis}[width=0.4\textwidth, axis equal image, scale only axis,  enlargelimits=false, axis on top]
				\addplot graphics[xmin=-0.2,xmax=1.2,ymin=0,ymax=1.4] {./figures/traub.pdf};
			\end{axis}
	\end{tikzpicture}}
	\caption{\small{\textcolor{red}{Canviar aquesta figura per l'espai de parametres del cas cubic de la seccio 4 en funcio de $\delta$ (no em treballat en aquest paper amb les families de les figures). Pregunta, per al punt critic a la preimatge de la conca, el numeric sempre surt sense cap bifurcacio, oi? El problema d'aqueste figures es que tindrem problemes amb la determinacie de l'arrel cuadrada}}}
	\label{fig:paramcubic}
\end{figure}
\end{comment}
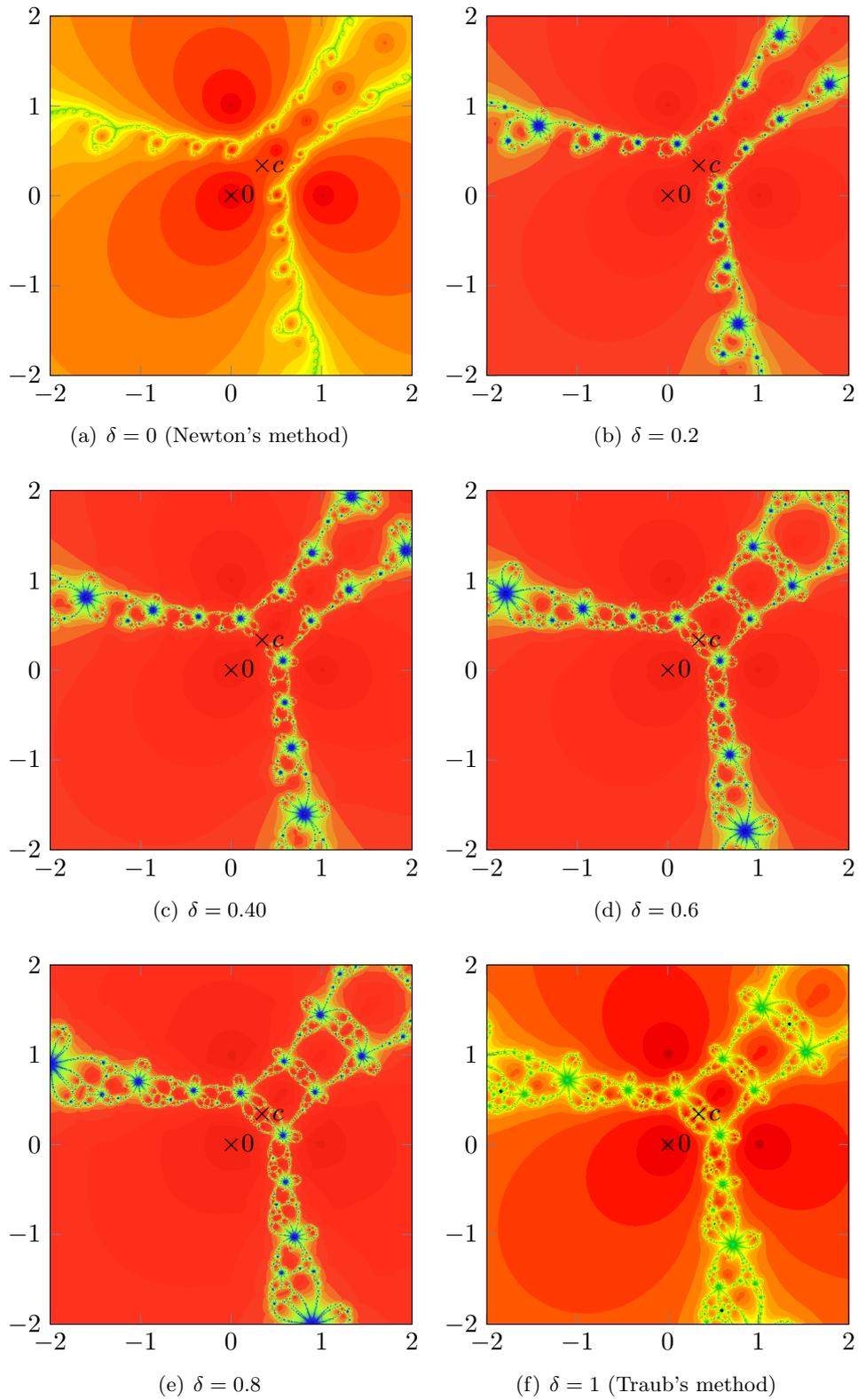
\begin{figure}[p]
	\centering
	\subfigure[ $\delta =0$ (Newton's method)   ] 
	{\begin{tikzpicture}
			\begin{axis}[width=0.4\textwidth, axis equal image, scale only axis,  enlargelimits=false, axis on top]
				\addplot graphics[xmin=-2,xmax=2,ymin=-2,ymax=2] {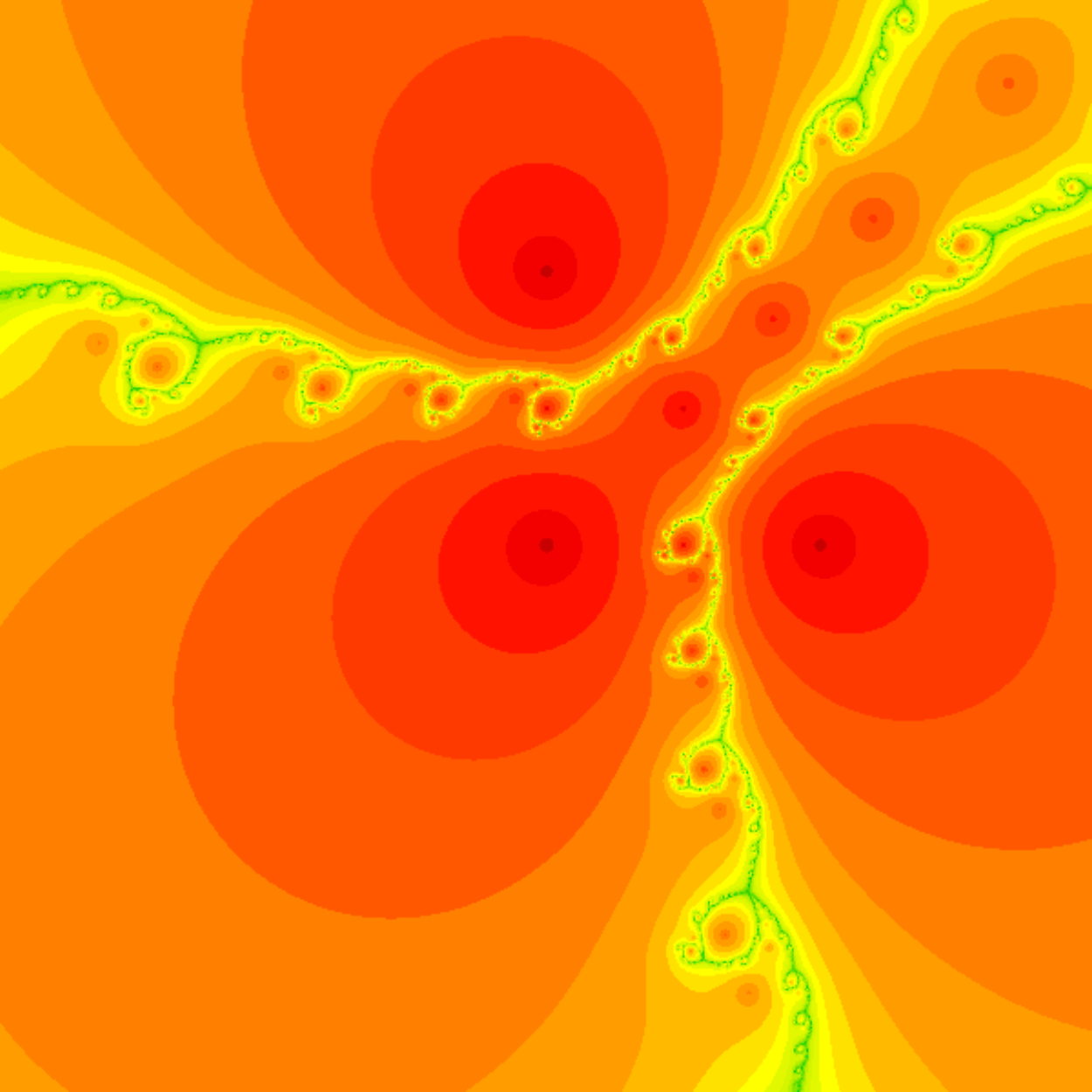};
			\end{axis}
	\end{tikzpicture}
	\put(-90.5,89){ $\times 0$}
		\put(-77.5,102){ $\times c$}
}
	\subfigure[ $\delta = 0.2$   ]{\begin{tikzpicture}
			\begin{axis}[width=0.4\textwidth, axis equal image, scale only axis,  enlargelimits=false, axis on top]
				\addplot graphics[xmin=-2,xmax=2,ymin=-2,ymax=2] {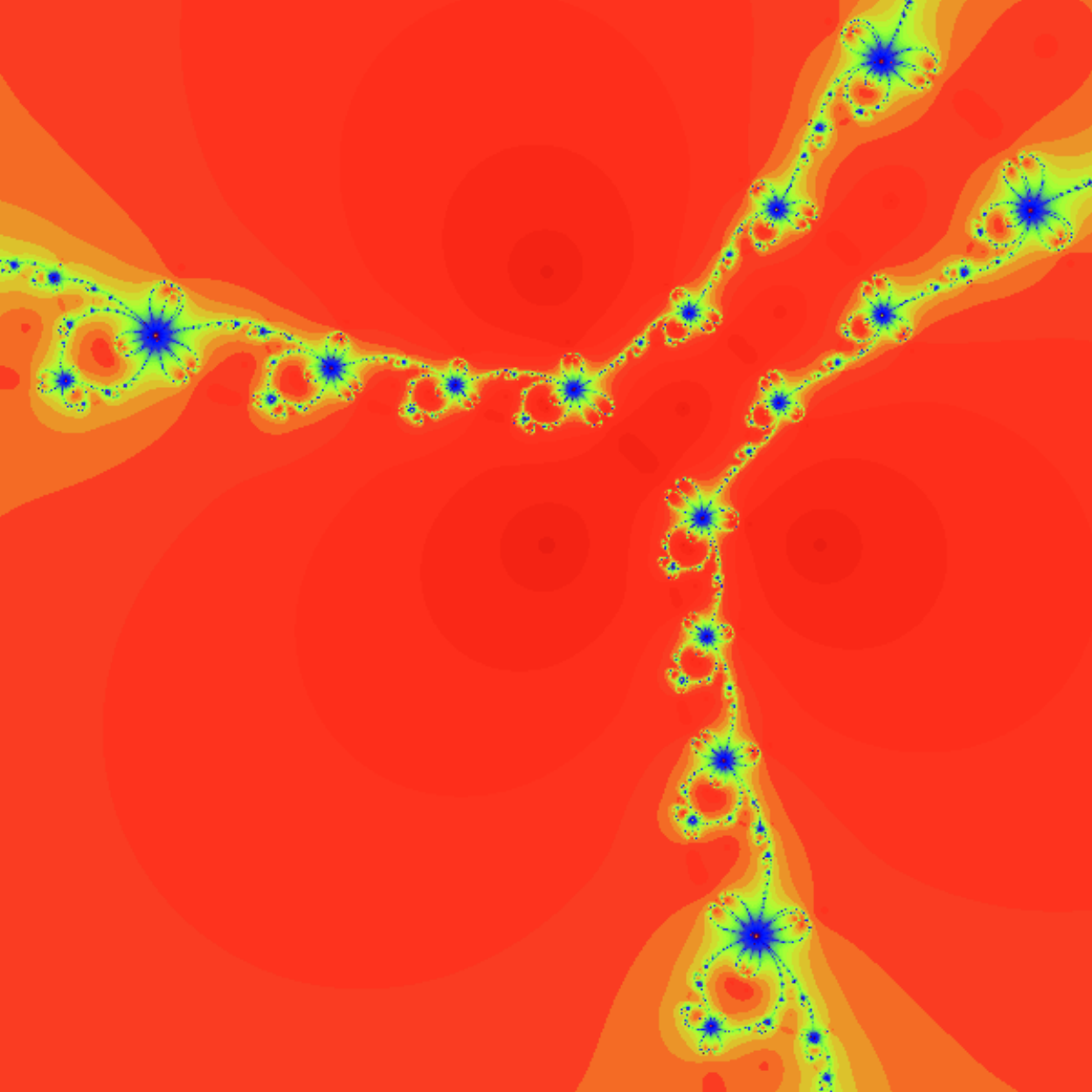};
			\end{axis}
	\end{tikzpicture}
\put(-90.5,89){ $\times 0$}
	\put(-77.5,102){ $\times c$}
}
	\subfigure[ $\delta = 0.40$   ] 
	{\begin{tikzpicture}
			\begin{axis}[width=0.4\textwidth, axis equal image, scale only axis,  enlargelimits=false, axis on top]
				\addplot graphics[xmin=-2,xmax=2,ymin=-2,ymax=2] {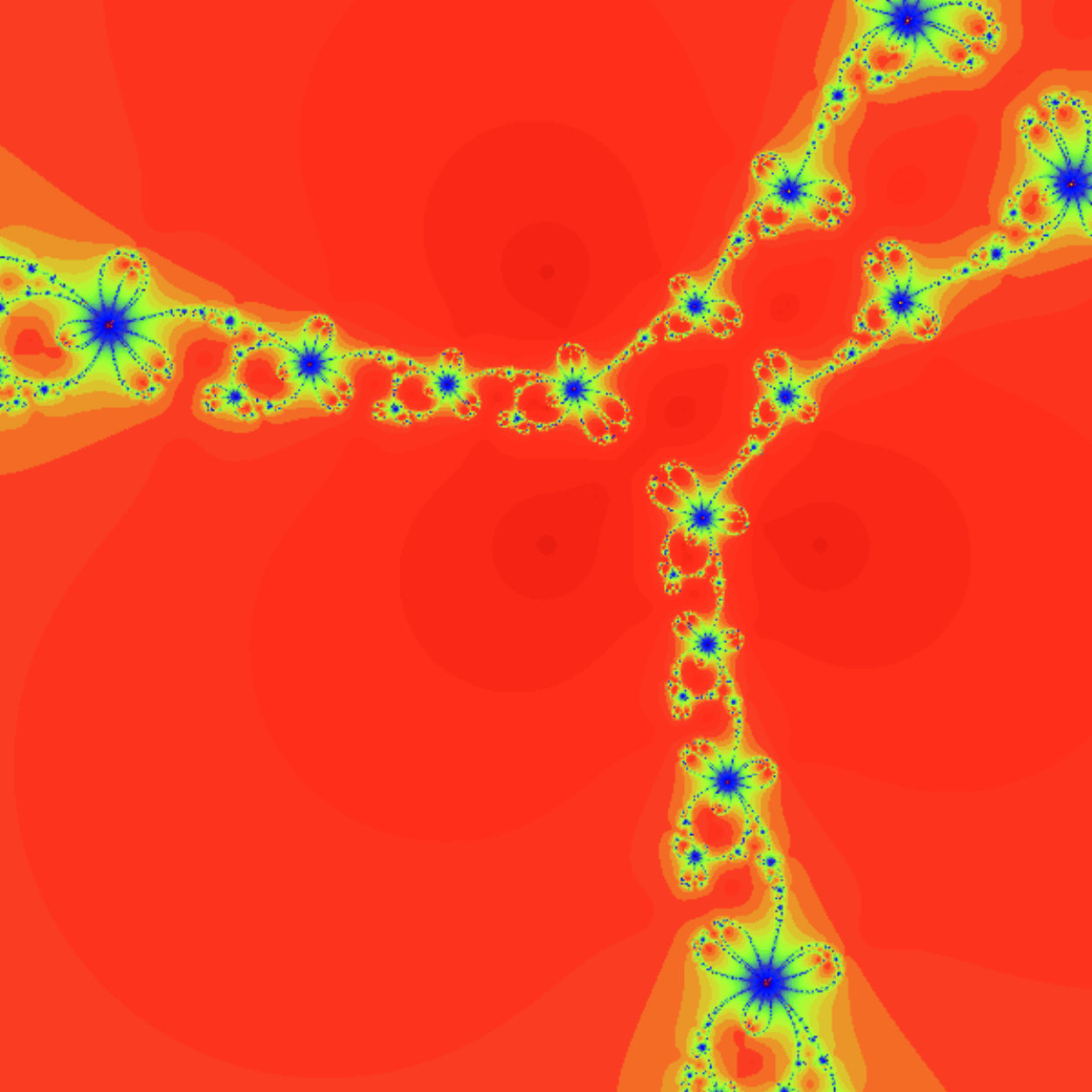};
			\end{axis}
	\end{tikzpicture}
\put(-90.5,89){ $\times 0$}
	\put(-77.5,102){ $\times c$}
}
	\subfigure[$\delta=0.6$   ]{\begin{tikzpicture}
			\begin{axis}[width=0.4\textwidth, axis equal image, scale only axis,  enlargelimits=false, axis on top]
				\addplot graphics[xmin=-2,xmax=2,ymin=-2,ymax=2] {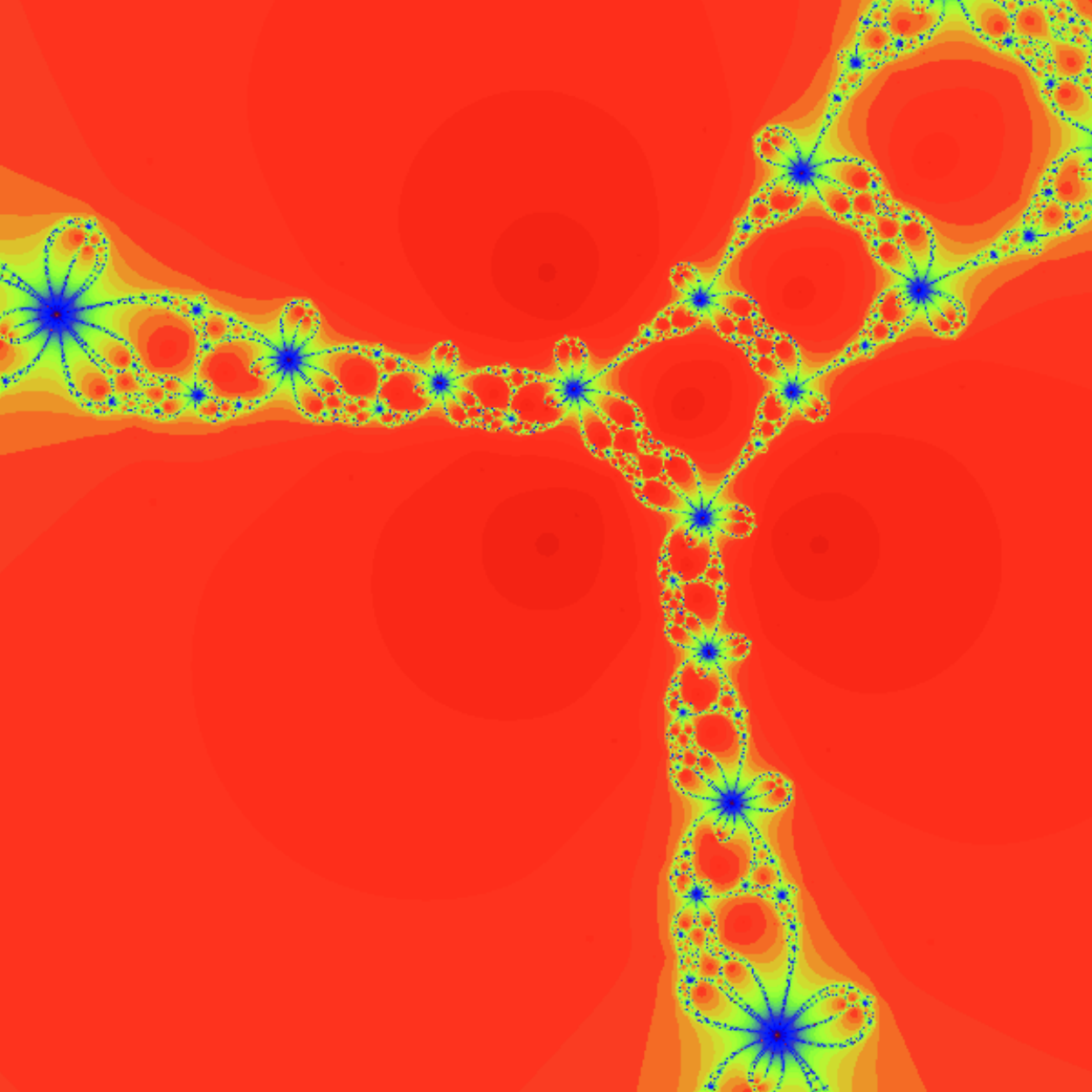};
			\end{axis}
			\end{tikzpicture}
	\put(-90.5,89){ $\times 0$}	
		\put(-77.5,102){ $\times c$}
	}
			\subfigure[$\delta=0.8$   ]{\begin{tikzpicture}
			\begin{axis}[width=0.4\textwidth, axis equal image, scale only axis,  enlargelimits=false, axis on top]
				\addplot graphics[xmin=-2,xmax=2,ymin=-2,ymax=2] {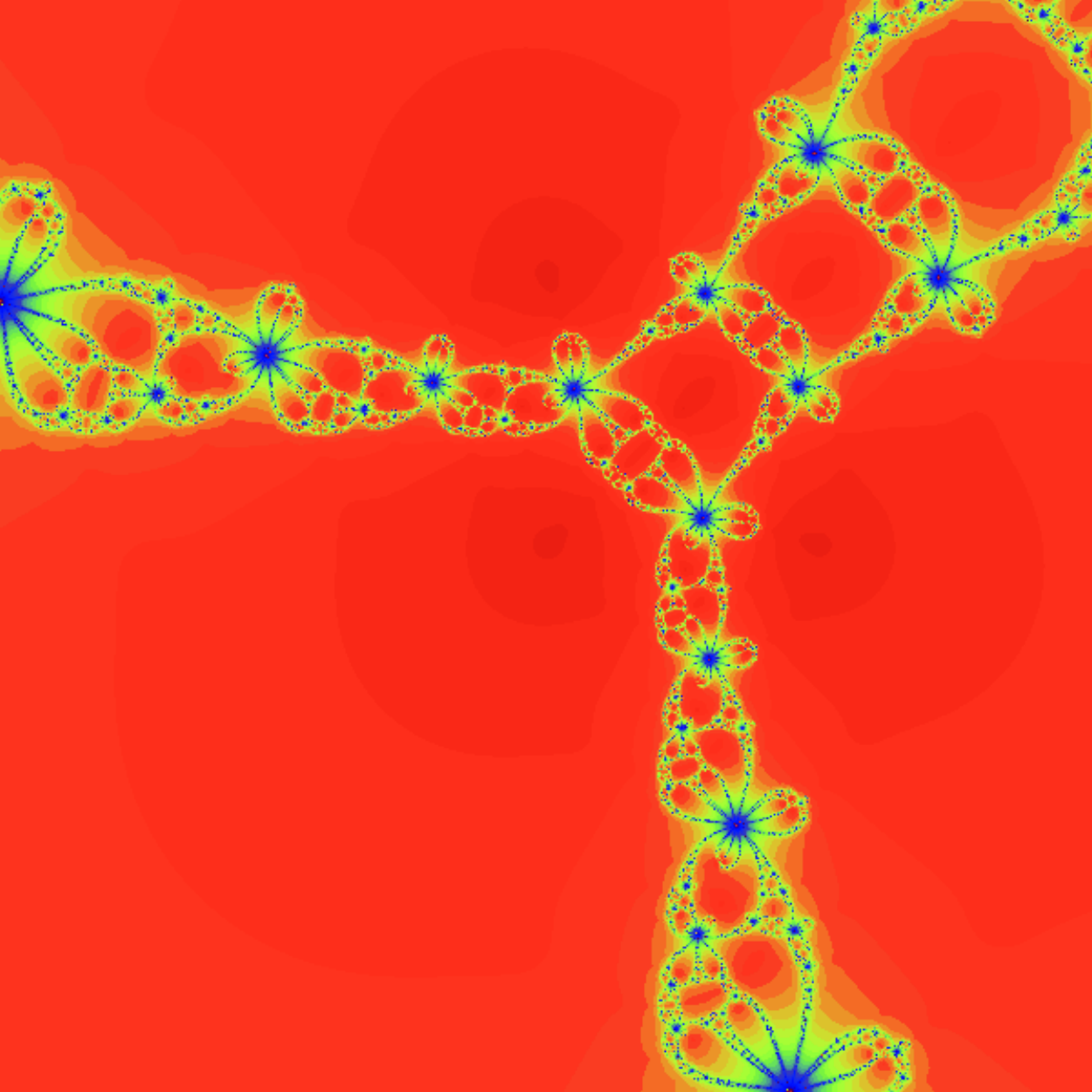};
			\end{axis}
			\end{tikzpicture}
	\put(-90.5,89){ $\times 0$}
		\put(-77.5,102){ $\times c$}	
	}
	\subfigure[ $\delta =1$ (Traub's method)   ]{\begin{tikzpicture}
			\begin{axis}[width=0.4\textwidth, axis equal image, scale only axis,  enlargelimits=false, axis on top]
				\addplot graphics[xmin=-2,xmax=2,ymin=-2,ymax=2] {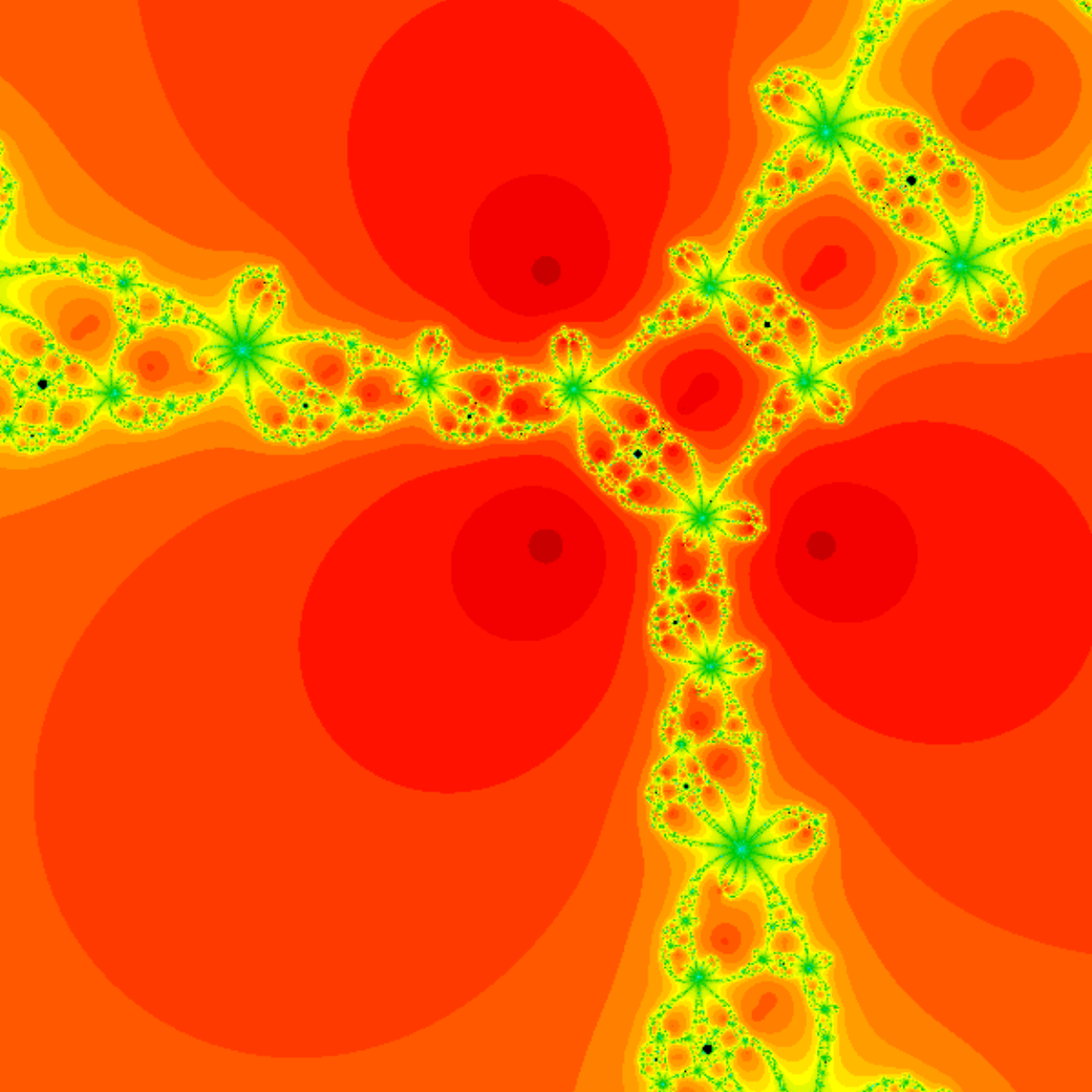};
			\end{axis}
	\end{tikzpicture}
\put(-90.5,89){ $\times 0$}
	\put(-77.5,102){ $\times c$}
}

	\caption{\small{Dynamical plane of damped Traub's method applied to $p(z)=z(z-1)(z-i)$.}}
	\label{fig:numericdinam}
	
\end{figure}

Numerical experiments  also indicate that when an access to $\infty$ is closed,  it is not due to the  critical points which appeared after perturbation, but rather due to a critical point which is a continuation of a free critical point of Newton's method (and already belongs to the immediate basin of attraction). In Figure~\ref{fig:numericdinam} we can see, step by step, how this process of closing one access to $\infty$ can happen. In Figure~\ref{fig:numericdinam}(a) we can observe the dynamical plane of Newton's method applied to $p(z)=z(z-1)(z-i)$. In this case the basins of attraction of $z=1$ and $z=i$ are unbounded and have a unique access to $\infty$ since they do not contain any free critical point. The basin of attraction of $z=0$ does contain a free critical point $c$ and has two accesses to $\infty$. Notice that, when applying Newton's method to polynomials, the number of accesses to $\infty$ from an immediate basin of attraction of a root is always equal to the number of free critical points that it contains plus 1 \cite{HowToNewton}. As $\delta$ grows from 0 to $1$ we can observe how the Julia set closes over the critical point $c$, erasing one of the accesses to $\infty$ (see Figure~\ref{fig:numericdinam} (f)).
Nevertheless, the \textit{skeleton} of the Julia set obtained for Traub's method ($\delta=1$) is still strongly related to the Julia set obtained for Newton's method. Following this idea, we believe that  if the immediate basin of attraction of a root has $d+1$ accesses to $\infty$ under Newton's map, then the $d$ free critical points it contains can be responsible for closing up to $d$ accesses to $\infty$ when the parameter moves from $\delta=0$ up to reaching Traub's method ($\delta=1$). However, this process would always leave an open access to $\infty$.

We have explained how damped Traub's method can help us understand the unboundedness of the immediate basins of attraction of Traub's method relating it to the dynamics of Newton's method. It is not obvious though that it can also be used to provide a better understanding of the simple connectivity. Indeed, there is a priori no obstruction in obtaining disconnected Julia sets after the singular perturbation for $|\delta|$ small, which could lead to multiply connected basins of attraction. An example of such behaviour could be Chebyshev-Halley family of root-finding algorithms (see \cite{gato}). These root-finding algorithm's depend on a parameter $\alpha$. When $\alpha$ converges to $\infty$ the operators converge to Newton's method. Despite that, if $|\alpha|$ is large enough the immediate basins of attraction of the roots may be disconnected (compare \cite{CCV20}).

\bibliographystyle{alpha}
\bibliography{biblio}

\begin{thebibliography}{Mc{M}94}

\bibitem[Bea91]{BeardonBook}
Alan~F. Beardon.
\newblock {\em Iteration of rational functions}, volume 132 of {\em Graduate
  Texts in Mathematics}.
\newblock Springer-Verlag, New York, NY, 1991.

\bibitem[BFJK14]{ConnectivityMero}
Krzysztof Bara{\'n}ski, N{\'u}ria Fagella, Xavier Jarque, and Bogus{\l}awa
  Karpi{\'n}ska.
\newblock On the connectivity of the {J}ulia sets of meromorphic functions.
\newblock {\em Invent. Math.}, 198(3):591--636, 2014.

\bibitem[Bla94]{Newton-Blanchard}
Paul Blanchard.
\newblock The dynamics of {N}ewton's method.
\newblock In {\em Complex dynamical systems ({C}incinnati, {OH}, 1994)},
  volume~49 of {\em Proc. Sympos. Appl. Math.}, pages 139--154. Amer. Math.
  Soc., Providence, RI, 1994.

\bibitem[Can17]{Can1}
Jordi Canela.
\newblock {S}ingular perturbations of {B}laschke products and connectivity of
  {F}atou components.
\newblock {\em Discrete Contin. Dyn. Syst.}, 37(7):3567--3585, 2017.

\bibitem[CCV20]{CCV20}
Beatriz Campos, Jordi Canela, and Pura Vindel.
\newblock Connectivity of the julia set for the chebyshev-halley family on
  degree n polynomials.
\newblock {\em Communications in Nonlinear Science and Numerical Simulation},
  82:105026, 2020.

\bibitem[CFT16]{DampedTraubQuadratic}
Alicia Cordero, Alfredo Ferrero, and Juan~R. Torregrosa.
\newblock Damped {T}raub's method: convergence and stability.
\newblock {\em Math. Comput. Simulation}, 119:57--68, 2016.

\bibitem[CG93]{CarlesonGamelinBook}
Lennart Carleson and Theodore~W. Gamelin.
\newblock {\em Complex dynamics}.
\newblock Universitext: Tracts in Mathematics. Springer-Verlag, New York, NY,
  1993.

\bibitem[CTV13]{gato}
Alicia Cordero, Juan~R. Torregrosa, and Pura Vindel.
\newblock Dynamics of a family of {C}hebyshev-{H}alley type methods.
\newblock {\em Appl. Math. Comput.}, 219(16):8568--8583, 2013.

\bibitem[DLU05]{Trichotomy}
Robert~L. Devaney, Daniel~M. Look, and David Uminsky.
\newblock The escape trichotomy for singularly perturbed rational maps.
\newblock {\em Indiana University Mathematics Journal}, 54(6):1621--1634, 2005.

\bibitem[Hea88]{CombinatoricsNewton}
Janet~E. Head.
\newblock The combinatorics of {N}ewton's method for cubic polynomials.
\newblock Ph.D. thesis, Cornell Univ., Ithacam N.Y., 1988.

\bibitem[HSS01]{HowToNewton}
John Hubbard, Dierk Schleicher, and Scott Sutherland.
\newblock How to find all roots of complex polynomials by {N}ewton's method.
\newblock {\em Invent. Math.}, 146(1):1--33, 2001.

\bibitem[McM87]{FamiliesRational}
Curtis~T. McMullen.
\newblock Families of rational maps and iterative root-finding algorithms.
\newblock {\em Ann. of Math. (2)}, 125(3):467--493, 1987.

\bibitem[Mc{M}94]{McMullenBook}
Curtis~T. Mc{M}ullen.
\newblock {\em Complex dynamics and renormalization}, volume 135 of {\em Annals
  of Mathematics Studies}.
\newblock Princeton University Press, Princeton, NJ, 1994.

\bibitem[Mil06]{MilnorBook}
John Milnor.
\newblock {\em Dynamics in one complex variable}, volume 160 of {\em Annals of
  Mathematics Studies}.
\newblock Princeton University Press, Princeton, NJ, third edition, 2006.

\bibitem[Prz89]{Prz}
Feliks Przytycki.
\newblock Remarks on the simple connectedness of basins of sinks for iterations
  of rational maps.
\newblock In {\em Dynamical systems and ergodic theory ({W}arsaw, 1986)},
  volume~23 of {\em Banach Center Publ.}, pages 229--235. PWN, Warsaw, 1989.

\bibitem[Roe08]{LocalConnectivityNewton}
Pascale Roesch.
\newblock On local connectivity for the {J}ulia set of rational maps:
  {N}ewton's famous example.
\newblock {\em Ann. of Math. (2)}, 168(1):127--174, 2008.

\bibitem[Shi09]{ConnectivityJulia}
Mitsuhiro Shishikura.
\newblock The connectivity of the {J}ulia set and fixed points.
\newblock In {\em Complex dynamics}, pages 257--276. A K Peters, Wellesley, MA,
  2009.

\bibitem[Sma85]{Smale}
Steve Smale.
\newblock On the efficiency of algorithms of analysis.
\newblock {\em Bull. Amer. Math. Soc. (N.S.)}, 13(2):87--121, 1985.

\bibitem[Sul85]{NoWanderingTheorem}
Denis Sullivan.
\newblock Quasiconformal homeomorphisms and dynamics. {P}art {I}. {S}olution of
  the {F}atou-{J}ulia problem on wandering domains.
\newblock {\em Annals of Mathematics. Second Series}, 122(3):401--418, 1985.

\bibitem[Tan97]{TanLei2}
Lei Tan.
\newblock Branched coverings and cubic {N}ewton maps.
\newblock {\em Fund. Math.}, 154(3):207--260, 1997.

\bibitem[Tra64]{Traub_Book}
Joseph~F. Traub.
\newblock {\em Iterative Methods for the Solution of Equations}.
\newblock Prentice-Hall Series in Automatic Computation. Prentice-Hall, Inc.,
  1964.

\bibitem[VLCT18]{DampedTraubCubic}
J.~Enrique V\'{a}zquez-Lozano, Alicia Cordero, and Juan~R. Torregrosa.
\newblock Dynamical analysis on cubic polynomials of damped {T}raub's method
  for approximating multiple roots.
\newblock {\em Appl. Math. Comput.}, 328:82--99, 2018.

\end{thebibliography}

\end{document}